\documentclass[a4paper,11pt]{article}
\usepackage[T1]{fontenc}
\usepackage[utf8]{inputenc}
\usepackage[english]{babel}
\usepackage{lmodern}
\usepackage{dsfont}
\usepackage{amsmath, amsthm, amssymb} % AMS packages
\usepackage{authblk}
\usepackage{enumitem}
\usepackage{fullpage}
\usepackage{xcolor}
\usepackage{microtype} % enhanced typography
\usepackage[colorlinks=true]{hyperref}
\usepackage[normalem]{ulem}

\microtypesetup{babel=true} % microtype settings - REQUIRES MICROTYPE PACKAGE
\renewcommand{\qedsymbol}{$\blacksquare$} % change qed symbol to black square
\setenumerate[0]{label=(\alph*)} % change enumerate to a),b),...
\hypersetup{ % adjusting hyperref link colors
    colorlinks,
    linkcolor={red!50!black},
    citecolor={blue!50!black},
	urlcolor={blue!80!black}
}

\theoremstyle{plain}
\newtheorem{thm}{Theorem}[section]
\newtheorem*{claim}{Claim}
\newtheorem{prop}[thm]{Proposition}
\newtheorem{lem}[thm]{Lemma}
\newtheorem{cor}[thm]{Corollary}
\newtheorem{obs}[thm]{Observation}
\theoremstyle{remark}

\theoremstyle{definition}
\newtheorem{defn}[thm]{Definition}

\newcommand{\N}{\mathbb{N}} % non-negative integers
\newcommand{\R}{\mathbb{R}} % reals
\newcommand{\Z}{\mathbb{Z}} % integers

\newcommand{\calA}{\mathcal{A}}

\newcommand{\calC}{\mathcal{C}}

\newcommand{\calF}{\mathcal{F}}
\newcommand{\calG}{\mathcal{G}}
\newcommand{\calH}{\mathcal{H}}
\newcommand{\calI}{\mathcal{I}}
\newcommand{\calL}{\mathcal{L}}
\newcommand{\calS}{\mathcal{S}}
\newcommand{\calT}{\mathcal{T}}
\newcommand{\calU}{\mathcal{U}}
\newcommand{\calV}{\mathcal{V}}

\newcommand{\Delp}{\Delta_{v'}^{x'}}
\newcommand{\Delnp}{\Delta^{x'}_{v}}
\newcommand{\Del}{\Delta_v^x}
\newcommand{\MDel}{M^{x}_{v}}
\newcommand{\MDelp}{M^{x'}_{v}}

\newcommand{\STOP}{\texttt{STOP}}

\newcommand{\eps}{\varepsilon}

\title{The typical approximate structure of sets with bounded sumset}
\author[1]{Marcelo Campos\thanks{\href{mailto:marcelo.campos@impa.br}{\texttt{marcelo.campos@impa.br}}}}
\author[2]{Matthew Coulson\thanks{\href{mailto:matthew.coulson@uwaterloo.ca}{\texttt{matthew.coulson@uwaterloo.ca}}}}
\author[3]{Oriol Serra\thanks{\href{mailto:oriol.serra@upc.edu}{\texttt{oriol.serra@upc.edu}}}}
\author[4]{Maximilian W\"{o}tzel\thanks{\href{mailto:m.wotzel@uva.nl} {\texttt{m.wotzel@uva.nl}}}}
\affil[1]{Instituto de Matem\'atica Pura e Aplicada, Rio de Janeiro, RJ, 22460-320, Brazil}
\affil[2]{Department of Combinatorics \& Optimization, University of Waterloo, ON N2L 3G1, Canada}
\affil[3]{Department of Mathematics, Universitat Polit\`ecnica de Catalunya, 08034 Barcelona, Spain}
\affil[4]{Korteweg-de Vries Institute for Mathematics, Universiteit van Amsterdam, 1098 XG Amsterdam, The Netherlands}
\date{}

\begin{document}
\maketitle

\begin{abstract}
Let $A_1$ and $A_2$ be randomly chosen subsets of the first $n$ positive integers of cardinalities $s_2\geq s_1 = \Omega(s_2)$, such that their sumset $A_1+A_2$ has size $m$.
We show that asymptotically almost surely $A_1$ and $A_2$ are almost fully contained in arithmetic progressions $P_1$ and $P_2$ with the same common difference and cardinalities approximately $s_i m/(s_1+s_2)$.
We also prove a counting theorem for such pairs of sets in arbitrary abelian groups.
The results hold for $s_i = \omega(\log^3 n)$ and $s_1+s_2 \leq m = o(s_2/\log^3 n)$.
Our main tool is an asymmetric version of the method of hypergraph containers which was recently used by Campos to prove similar results in the special case $A=B$.


\medskip
\noindent\textbf{MSC2010 classes:} 11P70, 11B30, 05C65
\end{abstract}

\section{Introduction and Main Results}

The general framework of problems in additive combinatorics is to ask for the structure of a set $A$  subject to some additive constraint in an additive group. 
The celebrated theorem of Freiman~\cite{F1973} provides such a structural result in terms of arithmetic progressions when the sumset $A+A$ is small.
Classical results like the Kneser theorem in abelian groups or the Brunn--Minkowski inequality in Euclidean spaces naturally address a similar problem for the addition of \emph{distinct} sets $A$ and $B$.
Ruzsa's proof of Freiman's theorem does provide the same structural result for distinct sets $A, B$ with the same cardinality when their sumset is small. 
When the sumset $A+A$ is very small, then another theorem of Freiman shows that the set is  dense in one arithmetic progression, and this result has been also extended to distinct sets $A$ and $B$ by Lev and Smeliansky~\cite{LS1995} showing that both sets are dense in arithmetic progressions with the same common difference. 
Discrete versions of the Brunn--Minkowski inequality have also been addressed for distinct sets by Ruzsa~\cite{R1994b} and Gardner and Gronchi~\cite{GG2001}. 

Motivated by the Cameron-Erd\H{o}s conjecture on the number of sum-free sets in $[n]=\{1,2,\dots,n\}$, there has been a quest to analyze the \emph{typical} structure of sets satisfying some additive constraint.  
One of the most efficient techniques to address this problem is the method of hypergraph containers first introduced by Balogh, Morris and Samotij~\cite{BMS2015} and independently by Saxton and Thomason~\cite{ST2015}, which has been successfully applied to a number of problems of this flavour. 

In~\cite{ABMS2013} Alon, Balog, Morris and Samotij posed a conjecture on the number of sets $A$ of size $s\ge C\log n$ contained in $[n]$ which have sumset $|A+A|\le K|A|$, $K\le s/C$.
This was proved recently by Green and Morris~\cite{GM2016} for $K$ constant and recently extended by Campos~\cite{C2019} to $K=o(s/(\log n)^3)$. 
These counting results are naturally connected to the typical structure of these sets, showing that they are almost contained in an arithmetic progression of length $(1+o(1))Ks/2$. 
We build on the later work by Campos to adapt the result to distinct sets. 
Our main result is the following.

\begin{thm}\label{thm:easymain_structure}
Let $n\geq s_2\geq s_1 = \Omega(s_2)$ be integers and $m$ an integer satisfying \[s_1+s_2\leq m = o(s_2^2/(\log n)^3).\]
Then for almost all sets $X_1,X_2\subset [n]$ such that $|X_i|=s_i$ and $|X_1+X_2|\leq m$, there exist arithmetic progressions $P_1$ and $P_2$ with the same common difference of size $|P_i| = (1+o(1))s_i m/(s_1+s_2)$ and $|X_i\setminus P_i| = o(s_i)$.
\end{thm}

Theorem~\ref{thm:easymain_structure} extends the result by Campos which corresponds to the symmetric case  $X_1=X_2$.  
It is derived from Theorem~\ref{thm:approximatestructure} in Section~\ref{sec:mainproofs} which gives more detailed quantitative estimations of the asymptotics involved in the above statement.

Our proof requires that the cardinalities of the two sets $X_1, X_2$ are not arbitrarily far apart, with specifics being discussed in Section~\ref{sec:concluding}.
But this seems natural: Even if the condition $|X_1|=\Omega (|X_2|)$ can be weakened, it is not clear to us that a nontrivial structural result should hold when one set is much smaller than the other.
As an extreme example, consider the situation when $|X_2|=2$ and $X_1$ is a random set $X_1\subset [n]$ of size $|X_1| = \Omega(\log n)$ such that $|X_1+X_2|\leq |X_1|+c$ for some positive constant $c\geq 3$. 
Then $X_1$ is a union of at most $c$ arithmetic progressions. 
An arbitrary union of $c$ progressions will provide an example of a set $X_1$ with such a small sumset, and the probability that a random choice of $c$ progressions is well covered by a short single one is negligible.
Back-of-the-envelope calculations imply that an equivalent statement might hold, as long as the size of the smaller set is bounded.

We can also prove the following counting analogue to Theorem~\ref{thm:easymain_structure} for an arbitrary abelian group, which can be compared to Theorem~1.4 in~\cite{C2019}.
We need the following definition.
For an abelian group $G$ and any positive real number $t$ define $\beta(t) = \max\{|H| : H \leq G,\, |H|\leq t\}$.

\begin{thm}\label{thm:easymain_counting}
Let $G$ be an abelian group.
Let $n\geq s_2\geq s_1 = \Omega(s_2)$ be integers and $m$ an integer satisfying $s_1+s_2\leq m = o(s_2^2(\log s_2)^{-4}(\log n)^{-3})$.
Then for any $F_1,F_2\subset G$ with $|F_i|=n$, the number of pairs of sets $(X_1,X_2)\in 2^{F_1}\times 2^{F_2}$ such that $|X_i|=s_i$ and $|X_1+X_2|\leq m$ is at most \[2^{o(s_2)}\binom{\frac{s_1}{s_1+s_2}(m+\beta)}{s_1}\binom{\frac{s_2}{s_1+s_2}(m+\beta)}{s_2},\] where $\beta = \beta((1+o(1))m)$.
\end{thm}

In groups that allow a result similar to Freiman's $3k-4$ theorem one can get rid of the $\log s_2$ term in the upper bound of $m$ in Theorem~\ref{thm:easymain_counting}.
Specific examples would be the integers $G=\Z$ or the integers modulo some prime $p$, that is, $G=\Z/p\Z$.
An example discussed in~\cite{C2019} that can easily be adapted to the asymmetric case shows that the range of $m$ for which Theorem~\ref{thm:easymain_counting} holds cannot be improved to $m = \Omega(s_2^2/\log n)$.
It is not clear whether the same holds true for Theorem~\ref{thm:easymain_structure}, and an interesting question would be to investigate whether it might be true for any $m=o(s_2^2)$.
Furthermore, note that as will become apparent in the proofs, the requirement $s_i=\omega((\log n)^3)$ mentioned in the abstract is not used directly.
Rather, it exists because otherwise the statements in Theorems~\ref{thm:easymain_structure} and~\ref{thm:easymain_counting} would become vacuous, since no $m$ satisfying the bounds could exist.

The proofs of both Theorem~\ref{thm:easymain_structure} and Theorem~\ref{thm:easymain_counting} will make use of the method of hypergraph containers, building upon the approach by Campos~\cite{C2019}.
The usual structure of this method is to essentially have two separate ingredients, the first being a result about independent sets in hypergraphs following certain degree conditions, and the second being supersaturation and stability results.
Our proof will also follow along these lines.

To that end, in Section~\ref{sec:containers} we will present with Theorem~\ref{thm:container} a new variant of the asymmetric container lemma introduced by Morris, Samotij and Saxton~\cite{MSS2018}, which is extended to a multipartite version. 
This extension does not require new ideas but the verification of the statement involves some technical issues so that for completeness we include a full proof in Section~\ref{sec:containerlemmaproof}. 
The main focus of Section~\ref{sec:containers} will instead be the application of this multipartite version of the container lemma to prove Theorem~\ref{thm:containerfamily}, a statement which provides a relatively small family of sets which essentially contain sets whose product has bounded cardinality in a (not necessarily abelian) group. 
We chose to state Theorem~\ref{thm:containerfamily} in more generality than is required for our results, namely for general groups, as it follows in a natural way from the container lemma and can be used in further applications.
In Section~\ref{sec:stabilitysupersat} we will then develop the required supersaturation and stability results. 
Unlike previous works on these problems, which use variants of a theorem by Pollard~\cite{P1974} on averages of the representation function, for the stability results we use a robust version of Freiman's $3k-4$--theorem by Shao and Xu~\cite{SX2019}, which itself built on earlier work by Lev~\cite{L2000}. 
The $3k-4$--theorem of Freiman, or rather its asymmetric version by Lev and Smeliansky~\cite{LS1995}, states that if the sumset of two sets is sufficiently small, then the sets themselves are dense in arithmetic progressions. 
Shao and Xu extended the result to the sum of two sets along the edges of a sufficiently dense graph, its robust version.
The statement of Shao and Xu applies to two sets with the same cardinality, and both its statement and proof actually rely on this condition, a circumstance which prevents it from being used in our context directly. 
We give a similar robust version, Theorem~\ref{prop:freiman3k4robust}, that   can be seen as a truly asymmetric robust version of the Lev--Smeliansky theorem and may be of independent interest. 
It requires a different proof than the one given in~\cite{SX2019}, although a key result in both proofs is an appropriate robust version of Kneser's addition theorem, Theorem~\ref{prop:robustkneser} in the current work.
To our knowledge this is the first statement of this kind which is applicable to pairs of sets of completely unrelated cardinalities, and thus represents progress on a conjecture of Lev~\cite{L2000b}.
In Section~\ref{sec:mainproofs} we combine the two ingredients to prove Theorems~\ref{thm:approximatestructure} and~\ref{thm:counting_abeliangroups}, which are more technical versions of our main theorems.
The paper concludes with some final remarks in Section~\ref{sec:concluding}, both regarding specific aspects of the current results as well as discussions on some open problems.

\section{The Method of Hypergraph Containers}\label{sec:containers}

One of the key techniques used in~\cite{C2019} is based on an asymmetric version of the container lemma introduced by Morris, Samotij and Saxton~\cite{MSS2018} which allows for applications to forbidden structures with some sort of asymmetry. 
This asymmetry can be interpreted as considering bipartite hypergraphs.
The first key component to proving Theorem~\ref{thm:easymain_counting} is to further extend this bipartite version to a multipartite one as follows.

Let $r$ be a positive integer.
For an $r$-vector $x=(x_1,\dots,x_r)$ with nonnegative integer entries, we call an $r$-partite hypergraph $\calH$ with vertex set $V(\calH)=V_1\cup\dots\cup V_r$ \emph{$x$-bounded} if $|E\cap V_i|\leq x_i$ for every hyperedge $E\in E(\calH)$ and every $1\leq i\leq r$.
Denote by $\calI$ the family of independent sets of $\calH$, and for any $m\in \N$, define
\[\calI_{m}(\calH):=\left\{I: I\in\calI\text{ and } |I\cap V_r|\geq |V_r|-m\right\}.\] 
For a subset of vertices $L\subset V(\calH)$, the codegree is defined as $d_{\calH}(L)=|\{E\in E(\calH) : L\subset E\}|$.
Also, given a vector $v=(v_1,v_2,\ldots,v_r)\in \Z^r$, denote \[\Delta_v(\calH):=\max\{d_\calH(L): L\subset V(\calH),~|L\cap V_i|=v_i,~1\leq i\leq r\}.\]
Finally, for any vector $y$, $|y|$ will denote its $1$-norm $\sum |y_i|$.
\begin{thm}
\label{thm:container}
For all non-negative integers $r,r_0$ and each $R> 0$ the following holds. Suppose that $\calH$ is a non-empty $r$-partite $(1,\dots,1,r_0)$-bounded hypergraph with $V(\calH)=V_1\cup V_2\cup \ldots\cup V_r$, $m\in \N$, and $b,q$ are positive integers with $b\leq \min\{|V_1|,\dots,|V_{r-1}|,m\}$ and $q\leq m$, satisfying
\begin{equation}\label{deg cond}
\Delta_{y}(\calH)\leq R \left(m^{y_r}\prod_{i=1}^{r-1}|V_i|^{y_i}\right)^{-1}b^{|y|-1}e(\calH)\left(\frac{m}{q}\right)^{\mathds{1}[y_r>0]}
\end{equation}
for every vector $y=(y_1,y_2,\ldots,y_r)\in \left(\prod_{i=1}^{r-1}\{0,1\}\right)\times \{0,1,\ldots,r_0\}$. Then there exists a family $\mathcal{S}\subset \prod_{i=1}^r \binom{V_i}{\leq b}$ and functions $f\colon \mathcal{S}\to \prod_{i=1}^r 2^{V_i}$ and $g\colon \calI_m(\calH)\to \mathcal{S}$, such that, letting $\delta=2^{-(r_0+r-1)(2r_0+r)}R^{-1}$, the following statements hold.
\begin{enumerate}[label=(\roman*)]
\item If $f(g(I))=(A_1,A_2,\ldots,A_r)$ with $A_i\subset V_i$, then $I\cap V_i\subset A_i$ for all $1\leq i\leq r$.\label{item:container-1}
\item  For every $(A_1,A_2,\ldots,A_r)\in f(\mathcal{S})$, either $|A_i|\leq (1-\delta)|V_i|$ for some $1\leq i< r$, or $|A_r|\leq |V_r|-\delta q$.\label{item:container-2}
\item If $g(I)=(S_1,S_2,\ldots,S_r)$ and $f(g(I))=(A_1,A_2,\ldots,A_r)$, then $S_i\subset I\cap V_i$ for all $1\leq i\leq r$. Furthermore, $|S_i|>0$ only if $|A_j|\leq (1-\delta)|V_j|$ for some $i\leq j< r$ or $|A_r|\leq |V_r|-\delta q$.\label{item:container-3}
\end{enumerate}
\end{thm}

The proof of Theorem~\ref{thm:container} is lengthy and very similar to the original asymmetric lemma of Morris, Samotij and Saxton~\cite{MSS2018}, but for the sake of completeness we present it in full in Section~\ref{sec:containerlemmaproof}.
For now, let us see how to successively apply it to construct the container family we want to work with.
We will be making use of the following hypergraph construction.
For a group $G$ and finite subsets $A,B,C\subset G$, define the $3$-partite and $(1,1,1)$-bounded hypergraph $\calH(A,B,C)$ in the following way.
The vertex set is $A\sqcup B\sqcup C$ and $\{a,b,c\}$ is a hyperedge if $a\in A$, $b\in B$, $c\in C$, and $ab=c$.
Note that the sets $A,B,C$ need not actually be disjoint.

\begin{thm}\label{thm:containerfamily}
Let $G$ be a group and $\epsilon>0$.
Suppose $n,m$ and $s_1\leq s_2$ are integers such that $\log n \leq s_2 \leq m \leq s_1^2\log n$.
Let $F_1,F_2\subset G$ be subsets of $G$ of cardinality $|F_1|=|F_2|=n$ with product set $F=F_1F_2$.
Then there exists a family $\calA\subset 2^{F_1}\times 2^{F_2}\times 2^{F}$ of triples $(A_1,A_2,B)$ of size 
\begin{equation}\label{eq:container_family_size}
|\calA| \leq \exp\left(2^{20}\epsilon^{-2}\sqrt{m}(\log n)^{3/2} \right)
\end{equation} 
such that the following hold:
\begin{enumerate}
\item\label{item:cont_family_indCont} Let $X_1\subset F_1$, $X_2\subset F_2$ with $|X_i|=s_i$ and $|X_1X_2|\leq m$. Then there exists a triple $(A_1,A_2,B)\in\calA$ such that $X_1\subset A_1$, $X_2\subset A_2$ and $B\subset X_1X_2$.
\item\label{item:cont_family_contStruct} Let $(A_1,A_2,B)\in\calA$. Then either $\max\{|A_1|,|A_2|\}\leq m/\log n$, or there are at most $\epsilon^2 |A_1||A_2|$ tuples $(a_1,a_2)\in A_1\times A_2$ such that $a_1a_2\notin B$.
\end{enumerate}  
\end{thm}

Let us remark that the second part in property~\ref{item:cont_family_contStruct} states that the $3$-partite hypergraph $\calH(A_1,A_2,(A_1A_2)\setminus B)$ has few edges.
Before going into the specifics, let us also mention that in some sense proving Theorem~\ref{thm:containerfamily} from Theorem~\ref{thm:container} is a by now standard application of the container framework.
The main idea is that as long as the conditions of Theorem~\ref{thm:container} are met, one can continue to apply it to the induced sub-hypergraphs obtained from the previous application.
Since independent sets in the original hypergraph stay independent, one thus builds a tree of hypergraphs where the leaves will correspond to the final family one wants to obtain.

\begin{proof}
We will construct a rooted tree $\calT$ with root $\calH(F_1,F_2,F)$ and leaves $\calH(A_1,A_2,A_3)$ such that one of the following properties holds:
\begin{enumerate}[label=(\roman*)]
\item\label{item:cont_family_smallCont} $|A_1|<s_1$ or $|A_2|<s_2$,
\item\label{item:cont_family_smallMax} $\max\{|A_1|,|A_2|\}< m/\log n$,
\item\label{item:cont_family_smallLast} $|A_3|<|F|-m$, or
\item\label{item:cont_family_fewEdges} $\calH(A_1,A_2,A_3)$ has less than $\epsilon^2 |A_1||A_2|$ hyperedges.
\end{enumerate}
The end-goal is to essentially have $\calA$ be the subset of the leaves of $\calT$ that correspond to properties~\ref{item:cont_family_smallMax} and~\ref{item:cont_family_fewEdges}.

We construct $\calT$ in the following way.
Suppose we are given a vertex $\calH=\calH(V_1,V_2,V_3)$ of $\calT$ with $|V_i|\geq s_i$ for all $i=1,2$, $\max\{|V_1|,|V_2|\}\geq m/ \log n$, $|V_3|\geq |F|-m$ and $e(\calH)\geq \epsilon^2 |V_1||V_2|$. 
We apply Theorem~\ref{thm:container} with parameters $R=\epsilon^{-2}$, $q=m/\log n$ and $b=\sqrt{q}$.
Note that $b\leq \min s_i$ because of our upper bound on $m$.
Let us show that these choices indeed satisfy the codegree conditions of the container lemma.

Let $v=(v_1,v_2,v_3)\in \{0,1\}^3$.
The edges of the hypergraph are defined by a linear relation, so in addition to the $|v|$ fixed components of a hyperedge, one loses one additional degree of freedom to choose the remaining $h+1-|v|$ components.
In particular, for $|v|\in\{2,3\}$ we see that $\Delta_v(\calH)= 1$.
Because of the assumption that $e(\calH)\geq \epsilon^2 \prod |V_1||V_2|$, we see that in order to prove~\eqref{deg cond}, it suffices that the parameters are chosen such that
\begin{equation}\label{eq:cont_family_sufficient}
\Delta_v(\calH) \leq \epsilon^2Rq^{-v_3}|V_1|^{1-v_1}|V_2|^{1-v_2}b^{|v|-1}.  
\end{equation}
We now confirm that~\eqref{eq:cont_family_sufficient} indeed holds for all possible choices of $v$, beginning with $v=(1,1,1)$.
As stated before we have $\Delta_v(\calH)=1$ and we see that~\eqref{eq:cont_family_sufficient} simplifies to \[1 \leq \epsilon^2 R b^2/q\] which is true for our choices of parameters.

We next move to vectors $v$ such that $|v|=2$.
Again, just like before it holds that $\Delta_v(\calH)=1$ and it is straight-forward to check that~\eqref{eq:cont_family_sufficient} reduces to \[1 \leq \epsilon^2 R b \min\{|V_1|,|V_2|,q\}/q.\]
This inequality clearly holds when the minimum is $q$ (it is a weaker requirement than the $v=(1,1,1)$ case).
When the minimum is $\min\{|V_1|,|V_2|\}$ instead it also holds since $|V_1|,|V_2|\geq s_1\geq b = \sqrt{q}$.

Finally we check the vectors $v\in\{0,1\}^3$ with $|v|=1$.
If $v=(0,0,1)$ we see that $\Delta_v(\calH)\leq\min\{|V_1|,|V_2|\}$ and hence~\eqref{eq:cont_family_sufficient} reduces to \[\min\{|V_1|,|V_2|\} \leq \epsilon^2 R |V_1||V_2|/q\] which is true since $\max\{|V_1|,|V_2|\}\geq q$ by assumption.
For $v=(1,0,0)$ (resp. $v=(0,1,0)$) it holds that $\Delta_v(\calH)\leq |V_2|$ (resp. $\leq |V_1|$) and so~\eqref{eq:cont_family_sufficient} reduces to \[|V_i| \leq \epsilon^2 R |V_i|\qquad\text{for $i=1,2$,}\] which again is true for our choice of parameters.

So we see that $R=\epsilon^{-2}$, $q=m/\log n$ and $b=\sqrt{q}$ are indeed valid choices.
Hence, by Theorem~\ref{thm:container}, there exists a family $\calC\subset 2^{V_1}\times 2^{V_2} \times 2^{V_3}$ of size at most 
\begin{equation}\label{eq:max_num_children}
|\calC| \leq \prod_{i=1}^{3} \binom{|V_i|}{\leq b}\leq b^3\binom{n^2}{b}\binom{n}{b}^2 \leq n^{4b} = \exp\left(4 \sqrt{m \log n}\right),
\end{equation}
such that for each $I\in \calI_m(\calH)$ there exist $(A_1,A_2,A_3)\in\calC$ with $I\cap V_i \subset A_i$ for all $i\in[3]$, and either $|A_i|\leq (1-\delta)|V_i|$ for some $i=1,2$, or $|A_3|\leq |V_3|-\delta q$, with $\delta=\epsilon^2 2^{-15}$.
For each $(A_1,A_2,A_3)\in\calC$, add $\calH(A_1,A_2,A_3)$ as a child of $\calH$ in $\calT$.
In order to bound the number of leaves of $\calT$, we will first bound its height.
\begin{claim}\label{claim:treedepth}
The tree $\calT$ has height at most $d = 2^{18}\epsilon^{-2}\log n $.
\end{claim}
\begin{proof}
Suppose $\calH(A_1,A_2,A_3)$ is a vertex of $\calT$ of depth $d$.
Recall that after each application of Theorem~\ref{thm:container}, one component shrunk, hence after $d$ applications one of them shrunk at least $d/3$ times.
Since we started at $\calH(F_1,F_2,F)$ and $\delta=2^{-15}\epsilon^2$, either \[|A_3| \leq |F| - \frac{d\delta q}{3} = |F| - \frac{d\epsilon^2 m}{3\cdot 2^{15}\log n} < |F| - m,\] or for one $i=1,2$, \[|A_i| \leq (1-\delta)^{d/3}n \leq \exp(-\delta d/3)n< 1,\] and so this vertex has no children.
\end{proof}
We will now define the family $\calA$ formally.
If $\calL$ is the set of leaves of $\calT$, let \[\calA = \left\{(A_1,A_2,B) : \begin{array}{c}\calH(A_1,A_2,F\setminus B)\in \calL,\\|A_1|\geq s_1,\, |A_2|\geq s_2\text{ and }|B|\leq m\end{array}\right\}.\]
Since every tuple $(A_1,A_2,B)$ in this family corresponds to a leaf with $|A_i|\geq s_i$ for every $i=1,2$ and $|F\setminus B|\geq |F|-m$, we must have either $\max |A_i|\geq m/\log n$ or the corresponding hypergraph must have less than $\epsilon^2 |A_1||A_2|$ edges, that is, there are less than $\epsilon^2 |A_1||A_2|$ tuples $(a_1,a_2)\in A_1\times A_2$ such that $a_1a_2\notin B$. 
In any case, \ref{item:cont_family_contStruct} holds.

We now check that property~\ref{item:cont_family_indCont} holds.
Let $X_1,X_2$ with $X_i\subset F_i$, such that $|X_i|=s_i$ for $i=1,2$ and $|X_1X_2|\leq m$.
We see that $X_1\cup X_2 \cup (F\setminus (X_1X_2))$ is contained in $\calI_m\left(\calH(F_1,F_2,F)\right)$, and so by the properties of the containers there is a path in $\calT$ from the root to a leaf $\calH(A_1,A_2,F\setminus B)$ such that $X_i\subset A_i$ and $B\subset X_1X_2$.
By the size bounds for $X_i$ and $X_1X_2$ it is clear that this leaf must correspond to a triple in $\calA$.

Finally, the size of $\calA$ is at most the $d$th power of the maximal number of children of a vertex in $\calT$, and so by \eqref{eq:max_num_children} we see that \[|\calA| \leq \exp\left(4d \sqrt{m\log n}\right) \leq \exp\left(2^{20}\epsilon^{-2}\sqrt{m}(\log n)^{3/2}\right),\] and so \eqref{eq:container_family_size} holds.
\end{proof}

\section{Supersaturation and Stability Statements}\label{sec:stabilitysupersat}

For any abelian group $G$ and finite subsets $U,V\subset G$, define \[\alpha(U,V) = \max\{|V'| : V'\subset G,\, |V'|\leq |V|,\, |\langle V'\rangle|\leq |U|+|V|-|V'|\}.\]
Here, $\langle V' \rangle$ denotes the subgroup of $G$ generated by $V'$.
For an element $x\in G$, the number of \emph{representations} of $x$ in $U+V$ will be denoted by $r_{U,V}(x) = |\{(u,v)\in U\times V : u+v=x\}|$.

The following theorem is a variant, due to Campos~\cite{C2019}, of a generalization of Pollard's theorem proved by Hamidoune and Serra~\cite{HaSe2008}.

\begin{prop}[Theorem 3.2 in~\cite{C2019}]\label{prop:thm3_2}
Let $G$ be an abelian group, $t$ be a positive integer and $U,V\subset G$ satisfying $t\leq |V|\leq |U|<\infty$.
Then
\begin{equation}\label{eq:3_2}
\sum_{x\in U+V}\min(r_{U,V}(x),t) \geq t(|U|+|V|-t-\alpha),
\end{equation}
where $\alpha=\alpha(U,V)$.
\end{prop}

This can be used to get the following supersaturation result for sets of distinct sizes.
Recall that $\beta(t)$ denotes the cardinality of the largest subgroup of $G$ of order at most $t$.

\begin{cor}\label{cor:supersat}
Let $G$ be an abelian group, $A_1,A_2,B\subset G$ be finite and non-empty subsets of $G$ and $0<\epsilon<1/2$, and denote $\beta = \beta((1+4\eps)|B|)$.
If \[|A_1|+|A_2| \geq (1+2\epsilon)(|B|+\beta),\] then there are at least $\epsilon^2|A_1||A_2|$ pairs $(a_1,a_2)\in A_1\times A_2$ such that $a_1+a_2\not\in B$.
\end{cor}

\begin{proof}
Without loss of generality we can assume $|A_2|\geq |A_1|$.
If $|B|\leq (1-\epsilon^2)|A_2|$, then since $r_{A_1,A_2}(b)\leq |A_1|$ for every fixed $b\in B$, we have at least \[|A_1||A_2| - |B||A_1| \geq \epsilon^2 |A_1||A_2|\] pairs $(a_1,a_2)\in A_1\times A_2$ such that $a_1+a_2\not\in B$.
So we can assume $|B| > (1-\epsilon^2)|A_2|$, which also implies $|A_1|\geq \epsilon |A_2|$.
Now applying Proposition~\ref{prop:thm3_2} with $t=\epsilon |A_2|$, $U=A_2$ and $V=A_1$ gives us \[\sum_{x\in A_1+A_2}\min(r_{A_1,A_2}(x),\epsilon|A_2|) \geq \epsilon|A_2|(|A_1|+(1-\epsilon)|A_2|-\alpha)\]
and hence
\begin{equation}\label{eq:supersat_lb}
\sum_{x\in (A_1+A_2)\setminus B}\min(r_{A_1,A_2}(x),\epsilon|A_2|) \geq \epsilon|A_2|(|A_1|+(1-\epsilon)|A_2|-|B|-\alpha).
\end{equation}
We will show that \[\alpha \leq \max(\beta,\, |A_1|+|A_2|-(1+4\epsilon)|B|).\]
Indeed, suppose $A'\subset G$ satisfies $|A'|\leq |A_1|$ and $|\langle A'\rangle| \leq |A_1|+|A_2|-|A'|$.
If $|A'|\leq \beta$ we are done, so suppose $|A'|>\beta$, and hence $|\langle A'\rangle| \geq (1+4\epsilon)|B|$ by definition of $\beta$.
So $A'$ satisfies \[|A'| \leq |A_1|+|A_2|-|\langle A' \rangle| \leq |A_1|+|A_2|-(1+4\epsilon)|B|,\] which is what we wanted to show.
Now note that since $\epsilon < 1/2$ and $|B|\geq (1-\epsilon^2)|A_2|$ we have $4|B|> 3|A_2|$ and hence
\begin{equation*}
|A_1|+(1-\epsilon)|A_2|-|B|-(|A_1|+|A_2|-(1-4\epsilon)|B|) = 4\epsilon|B|-\epsilon|A_2| > 2\epsilon|A_2|> \epsilon|A_1|.
\end{equation*}
Similarly, since $|A_1|+|A_2|\geq (1+2\epsilon)(|B|+\beta)$ and $\epsilon<1/2$ we have 
\begin{align*}
|A_1|+(1-\epsilon)|A_2|-|B|-\beta &\geq |A_1|+(1-\epsilon)|A_2|-\frac{|A_1|+|A_2|}{1+2\epsilon}\\
&> |A_1|+(1-\epsilon)|A_2|-(1-\epsilon)(|A_1|+|A_2|)\\ 
&= \epsilon|A_1|.
\end{align*}
Hence~\eqref{eq:supersat_lb} implies \[\sum_{x\in (A_1+A_2)\setminus B} r_{A_1,A_2}(x) \geq \epsilon^2|A_1||A_2|.\]
\end{proof}

Next we are going to prove a statement giving us structural information on sets with (very) small sumsets that will later be applied to the containers obtained from Theorem~\ref{thm:containerfamily}.
As stated in the introduction, this comes in the form of Freiman's $3k-4$ theorem~\cite{F1973} and more specifically its asymmetric version due to Lev and Smeliansky~\cite{LS1995}.
In the context of the container approach, what is actually needed are so-called \emph{robust} versions of these results, that is, we want to obtain structural knowledge on sets $U,V$ in a group $G$ knowing only that $\{u+v : (u,v)\in\Gamma\}$ is small for some subset $\Gamma\subset U\times V$ that is large but not necessarily the full product.

Prior results of such type in the literature were mainly concerned with handling the case of sets having the same cardinality.
Note that this would even be an issue if we stipulated that $s_1=s_2$ in Theorem~\ref{thm:easymain_counting}.
The problem here is that while the pairs of sets that are counted may have the same size, the containers obtained via Theorem~\ref{thm:containerfamily} might differ slightly.
We modify the recently obtained robust version of Freiman's $3k-4$ theorem by Shao and Xu~\cite{SX2019}, which itself built on earlier work by Lev~\cite{L2000} to handle this, and obtain the following stability result.

For subsets $U,V$ of some abelian group $G$ and $\Gamma\subset U\times V$, denote by \[U\stackrel{\Gamma}{+} V=\{u+v \in G: (u,v)\in \Gamma\}\] the restricted sumset of $U$ and $V$.

\begin{thm}\label{prop:freiman3k4robust} 
Let $0<\epsilon <1/2$ and let $U,V$ be finite subsets of $\Z$ with $N=\min \{|U|, |V|\}\ge 3$ and $M=\max\{|U|,|V|\}\geq 2/\sqrt{\epsilon}$.
Let $\Gamma \subset U\times V$ with $|\Gamma|\ge (1-\epsilon)|U||V|$ and \[|U\stackrel{\Gamma}{+} V|= |U|+|V|+r.\]
If 
\begin{equation}\label{eq:3k4_rUB}
r<\frac{N}{2}-13\sqrt{\epsilon}M,
\end{equation}
then there are arithmetic progressions $P$ and $Q$ with the same common difference and lengths $|P|\le |U|+r+5\sqrt{\epsilon}M, |Q|\le |V|+r+5\sqrt{\epsilon}M$  such that $|P\cap U|\ge  (1-\sqrt{\epsilon})|U|$ and $|Q\cap V|\ge  (1-\sqrt{\epsilon})|V|$. 
\end{thm}

Let us start by discussing some aspects of this result.
The best upper bound for $r$ in~\eqref{eq:3k4_rUB} that one might possibly obtain is of the form $N$, the cardinality of the smaller set.
This is true even in the case of $\Gamma=U\times V$ where at least in the integers this is actually achieved by the results due to Lev and Smeliansky~\cite{LS1995} as well as Stanchescu~\cite{Stanchescu1996}.
Similarly in the special case of $N=M$, that is, both sets are of the same cardinality, Shao in~\cite{S2019} was recently able to establish a similar robust result.
He achieved this by first proving a new version of the \emph{Balog-Szemer\'edi-Gowers Theorem}~\cite{BaSz1994,Gowers1998} and then directly applying the classical results of Lev--Smeliansky and Stanchescu.

We will soon discuss different proof approaches (including ours) to statements like Theorem~\ref{prop:freiman3k4robust}.
Before that, we first apply it to get something needed for our particular application.

\begin{cor}\label{cor:relativestability}

Let $s_1\leq s_2$ be positive integers, and $0<\epsilon\leq 2^{-8}\left(\frac{s_1}{s_1+s_2}\right)^2$.
If $A_1,A_2,B\subset\Z$, such that $(1-\epsilon)|B|\leq |A_1| + |A_2|$ and $|A_i|\leq \left(\frac{s_i}{s_1+s_2}+2\sqrt{\epsilon}\right)|B|$ for $i=1,2$, then one of the following holds:
\begin{enumerate}
\item\label{item:relativestab_manyedges} There are at least $\epsilon^2|A_1||A_2|$ pairs $(a_1,a_2)\in A_1\times A_2$ such that $a_1+a_2\not\in B$.
\item\label{item:relativestab_closetoap} There are arithmetic progressions $P_1, P_2$ of length $|P_i| \leq \frac{s_i}{s_1+s_2}|B| + 4\sqrt{\epsilon}|B|$ with the same common difference such that $P_i$ contains all but at most $\epsilon|A_i|$ points of $A_i$.
\end{enumerate}
\end{cor}

\begin{proof}
Let $\Gamma = \{(a_1,a_2)\in A_1\times A_2 : a_1+a_2\in B\}$.
If $|\Gamma| < (1-\epsilon^2)|A_1||A_2|$ case~\ref{item:relativestab_manyedges} holds, so assume the converse.
For $0<\epsilon\leq 2^{-8}\left(\frac{s_i}{s_1+s_2}\right)^2$, it holds that \[\left(\frac{1}{2}+13\epsilon\right)\frac{s_i}{s_1+s_2} > \frac{s_i}{2(s_1+s_2)}\geq 8\sqrt{\epsilon} \geq 5\sqrt{\epsilon}+\frac{31}{2}\epsilon-26\epsilon^{3/2}+13\epsilon^2,\] which implies \[\frac{3}{2}\left(\frac{s_i}{s_1+s_2}-2\sqrt{\epsilon}-\epsilon\right)+(1-13\epsilon)\left(1-\frac{s_i}{s_1+s_2}-2\sqrt{\epsilon}-\epsilon\right)>1.\]
Since \[|A_i| \geq (1-\epsilon)|B|-\left(\frac{s_{3-i}}{s_1+s_2}+2\sqrt{\epsilon}\right)|B| = \left(\frac{s_i}{s_1+s_2}-\epsilon-2\sqrt{\epsilon}\right)|B|,\] it thus follows that
\begin{align*}
|A_1 \stackrel{\Gamma}{+}A_2| &\leq |B|\\
&\leq \frac{3}{2}\left(\frac{s_i}{s_1+s_2}-2\sqrt{\epsilon}-\epsilon\right)|B|+(1-13\epsilon)\left(1-\frac{s_i}{s_1+s_2}-2\sqrt{\epsilon}-\epsilon\right)|B|\\
&\leq \frac{3}{2}|A_i|+(1-13\epsilon)|A_{3-i}|.
\end{align*}
We can thus apply Theorem~\ref{prop:freiman3k4robust} with $\epsilon^2$ in place of $\epsilon$.
Note that since $s_2\geq s_1$, both $|A_1|$ and $|A_2|$ are upper bounded by $(\frac{s_2}{s_1+s_2}+2\sqrt{\epsilon})|B|$, and so the theorem implies that there exist arithmetic progressions $P_1$, $P_2$ with the same common difference of length 
\begin{align*}
|P_i| &\leq |A_1\overset{\Gamma}{+}A_2|-|A_{3-i}|+ 5\epsilon\left(\frac{s_2}{s_1+s_2}+2\sqrt{\epsilon}\right)|B|\\
&\leq |B|-\left(\frac{s_i}{s_1+s_2}-\epsilon-2\sqrt{\epsilon}\right)|B|+5\epsilon\left(\frac{s_2}{s_1+s_2}+2\sqrt{\epsilon}\right)|B|\\
&\leq \frac{s_i}{s_1+s_2}|B| + 4\sqrt{\epsilon}|B|
\end{align*}
such that \[|A_i\setminus P_i| = |A_i| - |A_i\cap P_i| \leq \epsilon|A_i|,\] so case~\ref{item:relativestab_closetoap} holds.
\end{proof}

\subsection*{Proof of Theorem~\ref{prop:freiman3k4robust}}

Before going into specifics, let us discuss the general idea behind our proof strategy of Theorem~\ref{prop:freiman3k4robust}, which follows the outline already used by Lev~\cite{L2000,L2000b} as well as Shao--Xu~\cite{SX2019}.
The first main insight is that as long as the set $\Gamma\subset U\times V$ satisfies some specific regularity conditions, one can map $U$ and $V$ to an appropriate cyclic group and lower bound the cardinality of $U\stackrel{\Gamma}{+}V$ not only by its projection.
Instead, one also gets an additional translate of the ``shorter'' set among $U$ and $V$.
One then obtains the desired results by investigating the projection more closely.
When this projection fills out all of the cyclic group, one uses Kneser's theorem in abelian groups~\cite{K1953}, a classical result in additive number theory.
In case the projection is not the full group one instead needs a robust version of Kneser's theorem.
Such a result was proved by Lev in~\cite{L2000} when $U=V$ (but even in nonabelian groups) and more recently by Shao and Xu~\cite{SX2019} when $U$ and $V$ are distinct but of the same cardinality.
As stated before, we follow this general outline and therefore prove a similar result for sets of distinct cardinalities, Theorem~\ref{prop:robustkneser}.
This theorem is of independent interest since in addition to being a key aspect in the proof of Theorem~\ref{prop:freiman3k4robust} it also has several other implications as already investigated in~\cite{SX2019}.

As mentioned before, this is not the only approach to establishing a statement like Theorem~\ref{prop:freiman3k4robust}.
In particular, Shao in~\cite{S2019} provided a framework that allows one to transfer an arbitrary structural result for sumsets of two sets of the same cardinality to a robust version with essentially the same quantitative bounds.
In fact, his statements also hold in the more general setting when both sets only differ by a fixed multiplicative constant, or alternatively for pairs of sets of arbitrary cardinalities when imposing stronger upper bounds on the size of the restricted sumset.
Another approach that lies in between these two strategies might be to only replace the use of the robust version of Kneser's theorem in Lev's approach.
As will become apparent in the sequel, this is currently the bottleneck for improving the quantitative aspects of~\eqref{eq:3k4_rUB} in Theorem~\ref{prop:freiman3k4robust}.

We will now go into the technical details of the proof, starting with defining the exact concept of regularity that we require.
We may think of $\Gamma\subset U\times V$ as a subgraph of the complete bipartite graph $K_{|U|,|V|}$ where the edges $(u,v)$ are colored by the element $c=u+v\in G$. The language of graphs will be handy, and in this way $d_\Gamma(x)$ and $N_\Gamma(x)$ will denote the degree (resp. neighborhood) of some vertex $x$.
Following Lev~\cite{L2000}, we introduce the following definition.

\begin{defn}\label{defn:regularrestrictedsumset}
Let $U,V$ be two finite sets in an abelian group and $K, s$ non-negative integers.
A subset $\Gamma \subset U\times V$ is \emph{$(K,s)$-regular} if the following two things are true:
\begin{enumerate}[label=(\roman*)]
\item $d_{\Gamma}(u)\ge |V|-s$ for each $u\in U$ and $d_{\Gamma}(v)\ge |U|-s$ for each $v\in V$.
\item For any $c\in U+V$ with $r_{U,V}(c)\geq K$, it holds that $c\in U\stackrel{\Gamma}{+}V$.
\end{enumerate}  
\end{defn}

We will now prove the previously mentioned robust version of Kneser's theorem which is effective even for sets of very different sizes.
This can be compared to Proposition~3.1 in~\cite{SX2019} as well as Theorem~2 in~\cite{L2000} in the abelian case, which have slightly better constants when the sets are close in size.

\begin{thm}\label{prop:robustkneser}
Let $U,V$ be finite sets in an abelian group $G$ with $|U|\le |V|$ and let $K,s$ be non-negative integers. 
If $\Gamma\subset U\times V$ is $(K,s)$-regular and $U\stackrel{\Gamma}{+} V\neq U+V,$ then \[|U\stackrel{\Gamma}{+} V|\ge |V|+\frac{|U|}{2} -K-2s.\]
\end{thm}

It would be interesting to know whether one can improve the $|U|/2$ term in Theorem~\ref{prop:robustkneser} to $|U|$ instead.
Even in their less general settings, both Lev and Shao--Xu were only able to replace the constant $1/2$ by the golden ratio.
As will become apparent in the proofs of Proposition~\ref{prop:almost1} this would result in an immediate quantitative improvement in the upper bound~\eqref{eq:3k4_rUB} in Theorem~\ref{prop:freiman3k4robust}.
It would also confirm a conjecture posed by Lev in~\cite{L2000}.
Note further that the non-triviality of such a statement clearly depends on the choices of $K$ and $s$.
For instance, whenever $K+s\geq|U|/2$, Theorem~\ref{prop:robustkneser} tells us that the restricted sumset is larger than $|V|-s$, which is a trivial lower bound obtained from $(K,s)$-regularity looking at only the neighborhood of a single element $u\in U$ in $\Gamma$.
So the interesting cases are when $K$ and $s$ are both small enough compared to $|U|$.
Conversely, very small $K$ and $s$ will almost certainly result in either the full and restricted sumsets being the same (and hence the statement holding vacuously) or both their cardinalities being closer to the product rather than the sum of the two sets.

\begin{proof}[Proof of Theorem~\ref{prop:robustkneser}]
Suppose the statement is false and take a counterexample that minimizes $|U|$, the cardinality of the smaller set.
Note that we can assume that the graph $\Gamma$ is \emph{saturated}, meaning that if some color $\sigma\in U+V$ is contained in $U\stackrel{\Gamma}{+}V$, then in fact all edges $(u,v)\in U\times V$ with $u+v=\sigma$ are contained in $\Gamma$.
We start by showing that for any $u,u'\in U$, the distance $u-u'$ has many representations in $V-V$.
To do this, note that since $\Gamma$ is $(K,s)$-regular, we have $|(u+V)\setminus (U\stackrel{\Gamma}{+}V)| \leq s$, and similarly if we replace $u$ by $u'$.
So \[|(u+V)\cup (u'+V)| \leq |U\stackrel{\Gamma}{+}V|+2s < |V|+\frac{|U|}{2}-K,\] which implies 
\begin{equation}\label{eq:kneser_smallsetdists}
r_{V,-V}(u-u') = |(u+V)\cap (u'+V)| = 2|V|-|(u+V)\cup (u'+V)| > |V|-\frac{|U|}{2}+K.
\end{equation}
Next, we will show that there are many popular colors in $\Gamma$.
For this, define the set $P$ by \[P=\left\{\sigma\in U\stackrel{\Gamma}{+}V : r_{U,V}(\sigma)\geq |U|/2\right\}.\]
Note that $\Gamma$ is saturated, so the number of representations in $U\stackrel{\Gamma}{+}V$ and $U+V$ is identical for every color that actually appears.

The motivation for studying $P$ will be that in fact, if an element $v\in V$ ``sees'' a single popular color, one can show that it actually has all of $U$ as its neighborhood and hence any missing color (which exists by assumption) is ``surrounded'' only by unpopular colors.
This will result in an induced subgraph of $\Gamma$ that represents a smaller counterexample which runs counter to the minimality assumption in the beginning.
Let us be more precise.

By $(K,s)$-regularity and the assumed upper bound on $|U\stackrel{\Gamma}{+}V|$, we have
\begin{align*}
|U|(|V|-s) &\leq |\Gamma|\\
&= \sum_{\sigma\in P}r_{U,V}(\sigma) + \sum_{\sigma\notin P}r_{U,V}(\sigma)\\
&< |P||U| + \left(|U\stackrel{\Gamma}{+}V|-|P|\right)\frac{|U|}{2}\\
&< |P|\frac{|U|}{2} + \left(|V|+\frac{|U|}{2}-K-2s\right)\frac{|U|}{2},
\end{align*}
which can be rearranged to get
\begin{equation}\label{eq:kneser_manypopdist}
|P| > |V|-\frac{|U|}{2}+K.
\end{equation}
Next, we will show that for every $v\in V$ with $(U+v)\cap P\neq\emptyset$, we in fact have 
\begin{equation}\label{eq:kneser_popcolornborhood}
U+v\subset U\stackrel{\Gamma}{+}V.
\end{equation}
To see this, suppose $u_0+v\in P$. 
By the definition of $P$, there is a set ${\mathcal P}_0\in U\times V$  with    $|{\mathcal P}_0|\ge |U|/2$ such that $u'+v'=u_0+v$ for each $(u',v')\in {\mathcal P}_0$. 
Note that since $u_0+v$ was fixed, the second components of these tuples are all pairwise distinct.
Let $u\in U$ be chosen arbitrarily. 
It follows from~\eqref{eq:kneser_smallsetdists} that there is a set ${\mathcal P}_1\in V\times V$ with $|{\mathcal P}_1|\ge |V|-\frac{|U|}{2}+K$  such that $v''-v'=u-u_0$ for each $(v',v'')\in {\mathcal P}_1$. 
Again, $u$ and $u_0$ are fixed, and hence the first components of these tuples are pairwise distinct as well.
Hence by inclusion-exclusion, there are at least $K$ pairs in ${\mathcal P}_0$ whose second coordinate coincides with the first coordinate of some pair in ${\mathcal P}_1$. 
Each two such pairs $((u',v') ,(v',v''))\in {\mathcal P}_0\times {\mathcal P}_1$ define the relation $(u'+v')+(v''-v')=(u_0+v)+(u-u_0)=v+u$, implying that  
\[r_{U,V}(u+v)\geq K,\] 
and so $u+v\in U\stackrel{\Gamma}{+}V$ by $(K,s)$-regularity. 
Since the choice of $u$ is arbitrary, this proves~\eqref{eq:kneser_popcolornborhood}.

Let $V'\subset V$ be the set of elements $v$ such that $(U+v)\cap P=\emptyset$.
Then \[\Gamma\cap (U\times (V\setminus V')) = U\times (V\setminus V'),\] so since $U\stackrel{\Gamma}{+}V\neq U+V,$ we must have $U\stackrel{\Gamma'}{+}V'\neq U+V'$, where $\Gamma'=\Gamma\cap (U\times V')$ is the induced subgraph of $\Gamma$ on $U\times V'$.
Furthermore, $\Gamma'$ is $(K,s)$-regular: Firstly, it is clear that at most $s$ edges are missing in the neighborhood of every vertex, since this was the case for $\Gamma$.
Secondly, suppose $x\in U+V'$ is an element such that $r_{U,V'}(x)\geq K$.
Then since $V'\subset V$, $r_{U,V}(x)\geq K$, and so $x\in U\stackrel{\Gamma}{+}V$.
Since $\Gamma$ was saturated, \emph{every} edge that represented $x$ was included, and hence $x\in U\stackrel{\Gamma'}{+}V'$.

Next we will show that $|U|>|V'|$.
To see this, first note that we have the trivial lower bound \[|U\stackrel{\Gamma'}{+}V'|\geq |V'|-s,\] by using $(K,s)$-regularity and looking at the neighborhood of a single vertex of $U$ in $\Gamma'$.
On the other hand, every color in $U\stackrel{\Gamma'}{+}V'$ is contained in $(U\stackrel{\Gamma}{+}V)\setminus P$, and by~\eqref{eq:kneser_manypopdist} we thus have \[|U\stackrel{\Gamma'}{+}V'| \leq |U\stackrel{\Gamma}{+}V| - |P| < |V|+\frac{|U|}{2}-2s-K-\left(|V|-\frac{|U|}{2}+K\right) = |U|-2s-2K.\]
Combining these inequalities implies \[|U|>|V'|+s+2K,\] so in particular $|U|>|V'|$.
Since $U$ and $V$ represented a counterexample that minimized the cardinality of the smaller set, we must have 
\begin{equation}\label{eq:kneser_induction}
|U\stackrel{\Gamma'}{+}V'| \geq |U|+\frac{|V'|}{2}-2s-K.
\end{equation}
But then by combining~\eqref{eq:kneser_manypopdist} and~\eqref{eq:kneser_induction}, \[|U\stackrel{\Gamma}{+}V| \geq |P| + |U\stackrel{\Gamma'}{+}V'| > |V|+\frac{|U|}{2}-2s+\frac{|V'|}{2} > |V|+\frac{|U|}{2}-2s-K,\] a contradiction.
\end{proof}

We next prove a technical intermediate result in the group of integers that is comparable to Theorem~1 in~\cite{L2000} due to Lev as well as Proposition~2.2 in~\cite{SX2019} due to Shao and Xu.
The proof follows essentially the same arguments, using Theorem~\ref{prop:robustkneser} instead of other robust versions of Kneser's theorem, as well as being slightly more careful about the distinct set sizes.
The proof consists of two parts.
First, as previously mentioned, one shows that $(K,s)$-regularity will imply that, after projecting the restricted sumset into an appropriate cyclic group, one can lower bound the size of the original sumset by its projection plus one copy of the shorter set.
Appropriate here means that it will be large enough to contain both sets without any new collisions, but small enough to still get an extra copy of the shorter set.
The second part then analyzes the projected restricted sumset more closely.

For a finite  set of integers $U$ we denote its convex hull by $[U]=[\min (U), \max (U)]$. 

\begin{prop}\label{prop:almost1} Let $U,V$ be two finite sets of integers. Assume that $\gcd (U\cup V)=1$ and that $[U]=[0,\ell], [V]=[0,\ell']$, where $\ell'\le \ell$. 
Let $n=\min \{|U|,|V|\}$. 

Let $K\ge 2$, $s\ge 0$ and let $\Gamma\subset U\times V$ be $(K,s)$-regular. 
Then,
\[
|U\stackrel{\Gamma}{+} V|\ge \begin{cases} \ell +|V|-2s, & \ell \le |U|+|V|-2K-2\\ |U|+|V|+\frac{n}{2}-4s-2K-2, & \ell > |U|+|V|-2K-2.\end{cases}
\]
\end{prop}

\begin{proof} Let $f:\Z\to \Z/\ell\Z$ be the canonical projection. We write $f(x)=\tilde{x}$ and a similar notation for images of sets. From $U,V, \Gamma$ we build the modular version $\tilde{U}, \tilde{V}, \tilde{\Gamma}$ in $\Z/\ell \Z$. We have $|\tilde{U}|\geq |U|-1$ and $|\tilde{V}|\ge |V|-1$.

\begin{claim} 
$\tilde{\Gamma}$ is $(2K,s)$-regular.
\end{claim}

\begin{proof}
The $s$ missing edges incident to each vertex in $\Gamma$ produce at most $s$ missing edges incident to $\tilde{x}$ in $\tilde{\Gamma}$. 
On the other hand, since $\ell =\max (U)\ge \max (V)$, the preimage of each color different from zero in $\tilde{U}\times \tilde{V}$ produces at most two colors in $U\times V$. Hence, every nonzero color appearing at least $2K$ times in $\tilde{U}\times \tilde{V}$ must be present in $\Gamma$ and therefore it must also be present in $\tilde{\Gamma}$. If $\tilde{0}$ appears more than $2K$ times in $\tilde{U}\times \tilde {V}$, since $0$ and $2\ell$ appear at most one time in $U\times V$ and $K\ge 2$, then $\ell$ must appear at least $K$ times and the color is in $\Gamma$ (and hence in $\tilde{\Gamma}$).    
Thus $\tilde{\Gamma}$ is $(2K,s)$--regular. 
\end{proof}

Now note that for every element $c \in N_\Gamma(0)\cap N_\Gamma(\ell)\subset V$, $c$ and $c+\ell$ are distinct elements in $U\stackrel{\Gamma}{+} V$, but are mapped to the same element in $\Z/\ell\Z$.
Since $|N_\Gamma(u)|\geq |V|-s$ for any $u\in U$ by $(K,s)$-regularity, using inclusion exclusion we see that
\begin{equation}\label{eq:mod}
|U\stackrel{\Gamma}{+} V| \ge  |\tilde{U}\stackrel{\tilde{\Gamma}}{+} \tilde{V}|+ |N_\Gamma(0)\cap N_\Gamma(\ell)| \geq |\tilde{U}\stackrel{\tilde{\Gamma}}{+} \tilde{V}| + |V|-2s.
\end{equation}
We are now going to analyze different cases and show that in all of these, Equation~\eqref{eq:mod} will imply the theorem statement.

\paragraph{Case 1: $\ell \leq |U|+|V|-2K-2$.}\,

\noindent In this case every $\tilde{x}\in \Z/\ell \Z$ appears in $\tilde{U}\times \tilde{V}$ at least
\begin{align*}
|\tilde{U}\cap (\tilde{x}-\tilde{V})|&=|\tilde{U}|+|\tilde{V}|-  |\tilde{U}\cup (\tilde{x}-\tilde{V})|\\
&\ge |\tilde{U}|+|\tilde{V}|-\ell\\
&\ge |U|+|V|-2-\ell\\
&\geq 2K,
\end{align*}
times, and hence it appears in $\tilde{\Gamma}$ by $(2K,s)$-regularity. Therefore, $\tilde{U}\stackrel{\tilde{\Gamma}}{+}\tilde{V}=\Z/\ell Z$ and \eqref{eq:mod} gives
$$
|U\stackrel{\Gamma}{+} V|\ge \ell +|V|-2s,
$$
as claimed.

\paragraph{Case 2: $\ell > |U|+|V|-2K-2$ and $\tilde{U}\stackrel{\tilde{\Gamma}}{+} \tilde{V}\neq \tilde{U}+\tilde{V}$.}\,

\noindent In this case we can apply Theorem~\ref{prop:robustkneser} and see that 
$$
|\tilde{U}\stackrel{\tilde{\Gamma}}{+} \tilde{V}|\ge \max\{|\tilde{U}|,|\tilde{V}|\}+\frac{n-1}{2}-2K-2s \geq |U|+\frac{n}{2}-2K-2s-\frac{3}{2},
$$
and \eqref{eq:mod} yields
$$
|U\stackrel{\Gamma}{+} V|\ge |U|+|V|+\frac{n}{2}-2K-4s-\frac{3}{2},
$$
as claimed. 

\paragraph{Case 3: $\ell > |U|+|V|-2K-2$ and $\tilde{U}\stackrel{\tilde{\Gamma}}{+} \tilde{V}= \tilde{U}+\tilde{V}$.}\,

\noindent Here there will essentially be two sub-cases.
The easier one is if the restricted sumset is large.
Specifically, if $| \tilde{U}+\tilde{V}|\ge |\tilde{U}|+(n-1)/2$, then again \eqref{eq:mod} yields
$$
|U\stackrel{\Gamma}{+} V|\ge |\tilde{U}+ \tilde{V}|+|V|-2s\ge  |U|+|V|+\frac{n}{2}-2s-\frac{3}{2},
$$  
and we are done.

Suppose now that $| \tilde{U}+\tilde{V}|< |\tilde{U}|+(n-1)/2$. 
Then Kneser's theorem~\cite{K1953} implies that there is a nonzero subgroup $H\le G$ such that $\tilde{U}+\tilde{V}+H=\tilde{U}+\tilde{V}$ and

\begin{equation}\label{eq:kneser}
|\tilde{U}+\tilde{V}|=| \tilde{U}+H|+|\tilde{V}+H|-|H|.  
\end{equation} 

If $H=\Z/\ell Z$, then \eqref{eq:mod} with our current hypothesis $\ell \ge |U|+|V|-2K-2$ gives the conclusion with room to spare. Suppose that $H$ is a proper subgroup. We now repeat the adaptation by Shao and Xu of Lev's argument. Let
\begin{align*}
C_1&=\{ c \in U\stackrel{\Gamma}{+} V: \tilde{c}\in \tilde{V}\},\\
C_2&=\{ c \in U\stackrel{\Gamma}{+} V: \tilde{c}\in (\tilde{V}+H)\setminus \tilde{V}\}, \text{ and}\\
C_3&=\{ c \in U\stackrel{\Gamma}{+} V: \tilde{c}\in (\tilde{U}+\tilde{V})\setminus (\tilde{V}+H)\},
\end{align*}
which are pairwise disjoint since $0\in (H\cap\tilde{U}\cap\tilde{V})$ implies $\tilde{V}\subset(\tilde{V}+H)\subset (\tilde{U}+\tilde{V})$.

Using the same argument that was used to justify \eqref{eq:mod} and noting that $\tilde{V}\subset \tilde{U}+\tilde{V}=\tilde{U}\stackrel{\tilde{\Gamma}}{+} \tilde{V}$, we have 
\begin{equation}\label{eq:C1_LB}
|C_1|\ge |\tilde{V}|+|V|-2s.
\end{equation}
For $C_2$, note that since $\tilde{U}\stackrel{\tilde{\Gamma}}{+}\tilde{V}=\tilde{U}+\tilde{V}$, every element in $(\tilde{V}+H)\setminus \tilde{V}\subset \tilde{U}+\tilde{V}$ has a preimage in $U\stackrel{\Gamma}{+} V$, so that
\begin{equation}\label{eq:C2_LB}
|C_2|\ge |\tilde{V}+H|-|\tilde{V}|.
\end{equation}
We start by showing that in certain cases we can obtain a good lower bound on $|\tilde{V}+H|$ and hence prove the theorem statement using only $C_1$ and $C_2$.
Since $\gcd (U\cup V)=1$ and $0\in U\cap V$, it cannot happen that both $\tilde{U}$ and $\tilde{V}$ are contained in a single coset of $H$ since we assumed that $H\neq \Z/\ell\Z$. 

We are now going to investigate different cases.
Suppose first that $\tilde{U}$ is contained in a single coset of $H$, that is $\tilde{U}+H=H$.
Again, recall that since $0\in \tilde{U}\cap\tilde{V}$ and the previous observation, this means that $\tilde{V}$ intersects at least two distinct cosets of $H$.
That is, we have \[|\tilde{V}+H|\ge 2|H|\geq |\tilde{U}|+n-1,\] so that \[|U\stackrel{\Gamma}{+}V|\ge |C_1|+|C_2|\ge |\tilde{V}+H|+|V|-2s\geq |U|+|V|+n-2s-2\]
and we are done. 

Assume now that $\tilde{U}+H\neq H$, that is, $\tilde{U}$ intersects at least two cosets of $H$.
This case will be more involved and in particular we will need $C_3$ to establish the theorem statement.
Let $N=|(\tilde{U}+\tilde{V})\setminus(\tilde{V}+H)|/|H|$ be the number of cosets of $H$ inside $\tilde{U}+\tilde{V} = \tilde{U}+\tilde{V}+H$ but outside $\tilde{V}+H$. 
By Kneser's theorem there is at least one such coset, say $\tilde{u}+\tilde{v}+H$ with $\tilde{u}\not\in H$. 
Note that this implies
\begin{equation}\label{eq:preimagesinC3}
f^{-1}(\tilde{u}+\tilde{v}+H)\cap (U\stackrel{\Gamma}{+} V)\subset C_3.
\end{equation}
Let $U'=f^{-1}(\tilde{u}+H)\cap U$ and $V'=f^{-1}(\tilde{v}+H)\cap V$ the elements of $U$ and $V$ that are projected into these respective cosets. 
Then $U'+V'\subset f^{-1}(\tilde{u}+\tilde{v}+H)\cap (U+V)$.
On the other hand we see that $U'+V'$ contains the $|U'|+|V'|-1$ pairwise distinct elements of the form \[\min(U')+V'\quad\text{and}\quad U'+\max(V').\]
Since $\Gamma$ is $(K,s)$-regular and the above elements only live in the neighborhood of two vertices, at most $2s$ of these elements can be missing from $U\stackrel{\Gamma}{+}V$, and hence
\begin{equation}\label{eq:a'b'}
|f^{-1}(\tilde{u}+\tilde{v}+H)\cap (U\stackrel{\Gamma}{+} V)|\ge |U'|+|V'|-2s-1.
\end{equation}
By inserting in \eqref{eq:a'b'} the estimates
$$
|H|-|U'|\le |(\tilde{U}+H)\setminus \tilde{U}|,\;\; |H|-|V'|\le |(\tilde{V}+H)\setminus \tilde{V}|,
$$
and using~\eqref{eq:kneser} we thus obtain
\begin{equation}\label{eq:singleLB}
\begin{aligned}
|f^{-1}(\tilde{u}+\tilde{v}+H)\cap (U\stackrel{\Gamma}{+} V)| &\geq 2|H|+|\tilde{U}|+|\tilde{V}|-|\tilde{U}+H|-|\tilde{V}+H|-2s-1\\
&= |H|+|\tilde{U}|+|\tilde{V}|-|\tilde{U}+\tilde{V}|-2s-1.
\end{aligned}
\end{equation}
We see that clearly, \eqref{eq:preimagesinC3} and~\eqref{eq:singleLB} hold for every of the $N$ pairwise distinct cosets $\tilde{u}+\tilde{v}+H$ outside of $\tilde{V}+H$.
Furthermore, we also have that the sets $f^{-1}(\tilde{u}+\tilde{v}+H)$ are disjoint for distinct cosets.
We can hence combine~\eqref{eq:preimagesinC3} and~\eqref{eq:singleLB} to lower bound the cardinality of $C_3$ as follows, using the definition of $N$ as well as the fact that $N\geq 1$.
\begin{equation}\label{eq:C3_LB}
\begin{aligned}
|C_3| &\ge N(|H|+|\tilde{U}|+|\tilde{V}|-|\tilde{U}+\tilde{V}|-2s-1)\\
&\ge N|H|+(|\tilde{U}|+|\tilde{V}|-|\tilde{U}+\tilde{V}|-2s-1)\\
&=|(\tilde{U}+\tilde{V})\setminus(\tilde{V}+H)|+(|\tilde{U}|+|\tilde{V}|-|\tilde{U}+\tilde{V}|-2s-1)\\
&=|\tilde{U}|+|\tilde{V}|-|\tilde{V}+H|-2s-1.
\end{aligned}
\end{equation}
Finally, combining~\eqref{eq:C1_LB},~\eqref{eq:C2_LB} and~\eqref{eq:C3_LB} yields \[|U\stackrel{\Gamma}{+}V|\ge |C_1|+|C_2|+|C_3|\ge |\tilde{U}|+|\tilde{V}|+|V|-4s-1\geq |U|+|V|+n-4s-3.\]
This completes the proof.
\end{proof}

Now we can proceed with the proof of Theorem~\ref{prop:freiman3k4robust}.
For a finite set $X\subset G$ we will denote by $\ell(X)$ the size of its convex hull, that is $\ell(X)=\max(X)-\min(X)+1$.

\begin{proof}[Proof of Theorem~\ref{prop:freiman3k4robust}]

We begin by extracting a large regular subgraph of $\Gamma$ in order to apply Proposition~\ref{prop:almost1}.
Let \[U'=\{u\in U : d_\Gamma(u)\geq (1-\sqrt{\epsilon})|V|\},\] and observe that since \[(1-\epsilon)|U||V|\leq |\Gamma| \leq |U'||V|+(1-\sqrt{\epsilon})(|U|-|U'|)|V|,\] it holds that $|U'|\geq (1-\sqrt{\epsilon})|U|$.
Similarly, if $V'$ is the set of $v\in V$ with $d_\Gamma(v)\geq (1-\sqrt{\epsilon})|U|$, we have $|V'|\geq (1-\sqrt{\epsilon})|V|$.
% Define $V'$ as the set of all neighbors of $U'$ in $V$, that is, \[V'=\{v\in V : (u,v)\in \Gamma \text{ for some }u\in U'\},\] and let $\Gamma_1 = \Gamma \cap U'\times V'$ be the restriction of $\Gamma$.
If $\Gamma_1 = \Gamma \cap (U'\times V')$ is the restriction of $\Gamma$, then for every $u\in U'$ we see that \[d_{\Gamma_1}(u)=|N_{\Gamma}(u)\cap V'| \geq d_\Gamma(u)+|V'|-|V|\geq |V'|-\sqrt{\epsilon}|V|.\]
Similarly, $d_{\Gamma_1}(v)\geq |U'|-\sqrt{\epsilon}|U|$ for every $v\in V'$.
We may assume that $[U']=[0,\ell(U')]$, $[V']= [0,\ell(V')]$ and $\gcd(U'\cup V')=1$.
Furthermore, without loss of generality assume that \[\ell(U')\geq\ell(V').\]
We are almost in the position to apply Proposition~\ref{prop:almost1}.
One reason why we should not do it yet is that, if the shorter set is also the smaller one we would run into issues.
To prevent this, we introduce the following notion.
For a set $X$ with $[X]=[0,\ell(X)]$, denote by $h(X)=\ell(X)-|X|+1$ the number of holes of $X$.
We next split the remainder of the proof into two cases depending on which of $U'$ and $V'$ contains more holes.

The intuition is that if the longer set also contains more holes, then necessarily the shorter set must be larger and hence a straightforward application of Proposition~\ref{prop:almost1} will leads us to our goal.
If on the other hand the shorter set has more holes, we instead apply Proposition~\ref{prop:almost1} only to an appropriately truncated segment and get a single copy of the remainder ``for free''.
Let us make this precise, starting with the easy case.

\paragraph{Case 1. $h(U')> h(V')$.} 
Set $K=s=\sqrt{\epsilon}M$, and define \[\Gamma'=\Gamma_1\cup\{(u,v)\in U'\times V' : r_{U',V'}(u+v)\geq K\}.\]
Note that by doing this we have added at most \[\frac{(U'\times V')\setminus \Gamma_1}{K}\leq \frac{(U\times V)\setminus \Gamma}{K}\leq \sqrt{\epsilon}N\] elements $x\in (U'\stackrel{\Gamma'}{+}V')\setminus (U'\stackrel{\Gamma_1}{+}V')$, and hence 
\begin{equation}\label{eq:gammaprimenottoolarge}
|U\stackrel{\Gamma}{+}V| \geq |U'\stackrel{\Gamma'}{+}V'| - \sqrt{\epsilon}N.
\end{equation}
Since $\Gamma'$ is $(K,s)$-regular by construction, we can apply Proposition~\ref{prop:almost1} to get 
\begin{equation}\label{eq:gammaprimeLB}
|U'\stackrel{\Gamma'}{+} V'|\ge 
\begin{cases} 
\ell(U') + |V'| - 2s, & \ell(U') \le |U'|+|V'|-2K-2\\ 
|U'|+|V'|+\frac{\min\{|U'|,|V'|\}}{2}-4s-2K-2, & \ell(U') > |U'|+|V'|-2K-2.
\end{cases}
\end{equation}
Note that by our lower bounds on $|U'|,|V'|$ and~\eqref{eq:gammaprimenottoolarge}, the second line of~\eqref{eq:gammaprimeLB} would imply 
\begin{align*}
|U\stackrel{\Gamma}{+}V|&\geq |U'\stackrel{\Gamma'}{+}V'|-\sqrt{\epsilon}N\\
&\geq (1-\sqrt{\epsilon})(|U|+|V|+N/2)-6\sqrt{\epsilon}M-2-\sqrt{\epsilon}N\\
&\geq |U|+|V|+\frac{N}{2}-11\sqrt{\epsilon}M,
\end{align*}
which violates our initial assumption on $|U\stackrel{\Gamma}{+}V|$, so the first case must hold.
In particular,
\begin{align*}
\ell(U') &\leq |U'\stackrel{\Gamma'}{+}V'|-|V'|+2\sqrt{\epsilon}M\\
&\leq |U\stackrel{\Gamma}{+}V|-(1-\sqrt{\epsilon})|V|+2\sqrt{\epsilon}M+\sqrt{\epsilon}N\\
&< |U|+r+4\sqrt{\epsilon}M,
\end{align*}
and similarly using $h(U')>h(V'),$ \[\ell(V') < \ell(U')+|V'|-|U'| < |V|+r+4\sqrt{\epsilon}M.\]
\paragraph{Case 2. $h(U')\leq h(V')$.} 
Define $U_1'=U'\cap[0,\ell(V')]$, $U_2'=U'\setminus U_1'$ and $\Gamma_1'=\Gamma\cap (U_1'\times V')$.
We see that by definition of $V'$ it holds that at most $\sqrt{\epsilon}|U|$ of the $|U'|$ elements in $U'+\max(V')$ are missing from $U'\stackrel{\Gamma_1}{+}V'$.
This implies
\begin{equation}\label{eq:splitsumset}
|U\stackrel{\Gamma}{+}V|\geq |U'\stackrel{\Gamma_1}{+}V'|\geq |U_1'\stackrel{\Gamma_1'}{+}V'|+|U_2'|-\sqrt{\epsilon}|U|.
\end{equation}
Since $h(U_1')\leq h(U')\leq h(V')$ and $\ell(U_1')=\ell(V')$ by construction, we have $|U_1'|\geq |V'|$.
Furthermore, by definition of $U'$, for every $u\in U_1'$ it holds that $d_{\Gamma_1'}(u)\geq |V'|-\sqrt{\epsilon}|V|$, and similarly, $d_{\Gamma_1'}(v)\geq |U_1'|-\sqrt{\epsilon}|U|$ for every $v\in V'$.
Setting $K=s=\sqrt{\epsilon}M$ and defining \[\Gamma'' = \Gamma_1' \cup \{(u,v)\in U_1'\times V' : r_{U_1',V'}(u+v)\geq K\},\] we again see that
\begin{equation}\label{eq:gammaprimeprimeNotTooLarge}
|U_1'\stackrel{\Gamma_1'}{+}V'|\geq |U_1'\stackrel{\Gamma''}{+}V'|-\sqrt{\epsilon}N.
\end{equation}
Again, $\Gamma''$ is $(K,s)$-regular by construction, and so applying Proposition~\ref{prop:almost1} we get 
\begin{equation}\label{eq:gammaprimeprimeLB}
|U_1'\stackrel{\Gamma''}{+} V'|\ge 
\begin{cases} 
\ell(V') + |U_1'| - 2s, & \ell(V') \le |U'|+|V'|-2K-2\\ 
\frac{3}{2}|V'|+|U_1'|-4s-2K-2, & \ell(V') > |U'|+|V'|-2K-2.
\end{cases}
\end{equation}
Putting together~\eqref{eq:gammaprimeprimeNotTooLarge} and~\eqref{eq:splitsumset}, the second line of this would imply
\begin{align*}
|U\stackrel{\Gamma}{+}V| &\geq |U_1'\stackrel{\Gamma''}{+} V'|+|U_2'|-\sqrt{\epsilon}|U|-\sqrt{\epsilon}N\\
&\geq \frac{3}{2}|V'|+|U'|-10\sqrt{\epsilon}M\\
&\geq |U|+|V|+\frac{N}{2}-13\sqrt{\epsilon}M,
\end{align*}
a contradiction to our initial assumption.
Hence the first case of~\eqref{eq:gammaprimeprimeLB} must hold, which implies
\begin{align*}
\ell(V') &\leq |U_1'\stackrel{\Gamma''}{+} V'|-|U_1'|+2\sqrt{\epsilon}M\\
&\leq |U\stackrel{\Gamma}{+}V|-|U'|+4\sqrt{\epsilon}M\\
&\leq |V|+r+5\sqrt{\epsilon}M.
\end{align*}
Since $h(U')\leq h(V')$, we also have \[\ell(U')\leq \ell(V')-|V'|+|U'|\leq |U|+r+5\sqrt{\epsilon}M.\]
This completes the proof.
\end{proof}

\section{Proofs of the Main Results}\label{sec:mainproofs}

In this section we will prove more technical versions of Theorems~\ref{thm:easymain_structure} and~\ref{thm:easymain_counting}.
We start with the structural result.

\begin{thm}\label{thm:approximatestructure}
Let $s_1,s_2,n$ be integers and $\alpha>0$ a fixed real number satisfying \[s_2 \geq s_1\geq 2^{10}\alpha^{-1}(s_1+s_2)^{11/12}(\log n)^{1/4},\] and let $m$ be an integer such that \[{(1+\alpha)(s_1+s_2)\leq m < 2^{-108}\alpha^{12}s_1^{12}(s_1+s_2)^{-10}(\log n)^{-3}.}\]
Suppose $X_1,X_2\subset[n]$ are two uniformly chosen random sets with $|X_1|=s_1$, $|X_2|=s_2$ and $|X_1+X_2|\leq m$.
With probability at least ${1-\exp(-2^{5}m^{1/6}(s_1+s_2)^{2/3}\sqrt{\log n})}$ the following holds: there are sets $T_i\subset X_i$ of size ${|T_i| \leq 2^{11}\alpha^{-1} m^{1/6}(s_1+s_2)^{2/3}\sqrt{\log n}}$, such that $X_i\setminus T_i$ is contained in an arithmetic progression $P_i$ of size \[\frac{s_im}{s_1+s_2}+2^6 m^{13/12}(s_1+s_2)^{-1/6}(\log n)^{1/4},\] where $P_1$ and $P_2$ have the same common difference.
\end{thm}

The presence of the $\alpha$ parameter is of a technical nature in the following sense.
If $m$ is very close to $s_1+s_2$, say $m=(1+o(1))(s_1+s_2)$, then some of the computations break down.
In particular it becomes difficult to apply a technical result, Lemma~\ref{lem:binom_prod_bound} which we require to prove Theorem~\ref{thm:approximatestructure}.
On the other hand, from a moral standpoint it should only help that the parameter $m$ is smaller.
In fact, in many situations such smaller values of $m$ can be handled directly by the asymmetric version of Freiman's $3k-4$ theorem, as detailed in the following proof of Theorem~\ref{thm:easymain_structure} using Theorem~\ref{thm:approximatestructure}.

\begin{proof}[Proof of Theorem~\ref{thm:easymain_structure}]
If $m=(1+o(1))(s_1+s_2)=s_1+s_2+o(s_1)$, we can apply an appropriate asymmetric version of Freiman's $3k-4$ theorem (see for instance~\cite{LS1995}) directly and see that any sets $X_1$, $X_2$ satisfying the theorem hypotheses are contained in arithmetic progressions $P_1$ and $P_2$ with the same common difference of size \[|P_i|=(1+o(1))s_i = (1+o(1))s_i m / (s_1+s_2).\] 
If on the other hand there exists some absolute constant $\alpha >0$ such that $m\geq (1+\alpha)(s_1+s_2)$, we can apply Theorem~\ref{thm:approximatestructure} instead.
\end{proof}

In order to prove Theorem~\ref{thm:approximatestructure}, we need the following bound on the product of two particular binomial coefficients which cannot be derived from available and more generic bounds. 
The proof is straightforward but rather lengthy and might distract somewhat from the main thrust of the paper, so we include it in Appendix~\ref{sec:largedeviationbound}.

\begin{lem}\label{lem:binom_prod_bound}
Let $m$, $s$ and $t$ be positive integers and let $1\geq \alpha>0$ such that $m\geq (1+\alpha)(s+t)$ and $s+t\geq 2^5 \alpha^{-1}$.
If $\epsilon>0$ satisfies \[\frac{2^{10}\min(s^2,t^2)}{(s+t)^2 m^2} \leq \epsilon \leq \frac{\alpha^2 \min(s^2,t^2)}{2^{10}(s+t)^2},\] then
\begin{equation}\label{eq:binom_prod_bound}
\binom{\left(\frac{t}{s+t}-2\sqrt{\epsilon}+2\epsilon\right)m}{t}\binom{\left(\frac{s}{s+t}+2\sqrt{\epsilon}\right)m}{s} \leq e^{-\epsilon(s+t)}\binom{\frac{sm}{s+t}}{s}\binom{\frac{tm}{s+t}}{t}.
\end{equation}
\end{lem}

We are now ready to prove Theorem~\ref{thm:approximatestructure}.

\begin{proof}[Proof of Theorem~\ref{thm:approximatestructure}]
The upper bound on $m$ in particular implies $m\leq s_1^2 \log n$, so let $\calA$ be the family obtained from Theorem~\ref{thm:containerfamily} applied with $G=\Z$, $h=2$, $F_1=F_2=[n]$ and $2^{-10}\alpha^2 s_1^2/(s_1+s_2)^2>\epsilon > 2^{10}s_1^2 (s_1+s_2)^{-2}m^{-2}$ to be specified later.
We claim that one of the following holds for every triple $(A_1,A_2,B)\in\calA$:
\begin{enumerate}
\item\label{item:container_small} $|A_1|+|A_2| \leq (1-\epsilon)m$,
\item\label{item:container_max_large} $|A_i| > \frac{s_im}{s_1+s_2}+2\sqrt{\epsilon}m$ for some $i\in\{1,2\}$, or
\item\label{item:container_close_to_ap} There are arithmetic progressions $P_1$, $P_2$ with the same common difference and sets $T_1$, $T_2$ such that $|P_i|\leq \frac{s_im}{s_1+s_2}+4\sqrt{\epsilon}m$, $|T_i|\leq \epsilon |A_i|$ and $A_i\setminus T_i\subseteq P_i$ for $i=1,2$.
\end{enumerate}
Note first that we always have $|A_1|+|A_2|\leq (1+2\epsilon)m$ since by Theorem~\ref{thm:containerfamily}\ref{item:cont_family_contStruct} applied with $G=\Z$, either there are at most $\epsilon^2 |A_1||A_2|$ pairs $(a_1,a_2)\in A_1\times A_2$ with $a_1+a_2\not\in B$, and hence Corollary~\ref{cor:supersat} together with $|B|\leq m$ gives the required upper bound on $|A_1|+|A_2|$, or $\max\{|A_1|,|A_2|\} < m/\log n$.
Suppose neither \ref{item:container_small} nor \ref{item:container_max_large} hold, then by Corollary~\ref{cor:relativestability}\ref{item:relativestab_closetoap} we see that \ref{item:container_close_to_ap} holds.

We will now count the number of pairs of sets $X_1,X_2$ of size $s_1$ and $s_2$ respectively, satisfying $|X_1+X_2|\leq m$ that do not have large intersections with arithmetic progressions in the sense of the theorem.
To do this, recall that by Theorem~\ref{thm:containerfamily}\ref{item:cont_family_indCont}, for any such pair, there exists a container triple $(A_1,A_2,B)\in\calA$ such that $X_i\subset A_i$.

We begin by giving an upper bound on the number of $X_1,X_2$ with containers satisfying property~\ref{item:container_small}, that is, $|A_1|+|A_2| \leq (1-\epsilon)m$.
In fact we will give an upper bound for all such pairs, irrespective of whether they violate the claimed structural statement.
Clearly there are at most $\sum_{\calA}^{(a)} \binom{|A_1|}{s_1}\binom{|A_2|}{s_2}$ of these.
By comparing $\binom{a-b}{c}\binom{b}{d}$ and $\binom{a-b-1}{c}\binom{b+1}{d}$ it is easy to check that an expression of this form has its maximum at $\binom{ca/(c+d)}{c}\binom{da/(c+d)}{d}$. 
So choosing ${\epsilon = 2^8 m^{1/6}(s_1+s_2)^{-1/3}\sqrt{\log n}< 2^{-10}\alpha^2 s_1^2(s_1+s_2)^{-2}}$ and using \eqref{eq:container_family_size}, we see that
\begin{equation}\label{eq:type_a_sets}
\begin{split}
\sum_{\calA}^{(a)}\binom{|A_1|}{s_1}\binom{|A_2|}{s_2} &\leq |\calA| \binom{(1-\epsilon)\frac{s_1m}{s_1+s_2}}{s_1}\binom{(1-\epsilon)\frac{s_2m}{s_1+s_2}}{s_2}\\
&\leq \exp(2^{21} \sqrt{m} \epsilon^{-2}(\log n)^{3/2} - \epsilon (s_1+s_2))\binom{\frac{s_1m}{s_1+s_2}}{s_1}\binom{\frac{s_2m}{s_1+s_2}}{s_2}\\
&\leq \exp(-2^7m^{1/6}(s_1+s_2)^{2/3}\sqrt{\log n})\binom{\frac{s_1m}{s_1+s_2}}{s_1}\binom{\frac{s_2m}{s_1+s_2}}{s_2}.
\end{split}
\end{equation}
We will now count pairs coming from containers of type~\ref{item:container_max_large}.
As in the last case we will count all pairs, not such those violating the theorem statement.
We will not make use of the fact that $s_2\geq s_1$ so suppose without loss of generality that \ref{item:container_max_large} holds for $i=2$.
Similar to the previous case, it suffices to give an upper bound for \[\sum_{\calA}^{(b)}\binom{|A_1|}{s_1}\binom{|A_2|}{s_2} \leq \sum_{\calA}^{(b)}\binom{\left(\frac{s_1}{s_1+s_2}+2\epsilon-2\sqrt{\epsilon}\right)m}{s_1} \binom{\left(\frac{s_2}{s_1+s_2}+2\sqrt{\epsilon}\right)m}{s_2}.\]
Noting that $\epsilon = 2^8m^{1/6}(s_1+s_2)^{-1/3}\sqrt{\log n} > 2^{10}s_1^2(s_1+s_2)^{-2}m^{-2}$ we can apply Lemma~\ref{lem:binom_prod_bound} and see that \[\binom{\left(\frac{s_1}{s_1+s_2}+2\epsilon-2\sqrt{\epsilon}\right)m}{s_1} \binom{\left(\frac{s_2}{s_1+s_2}+2\sqrt{\epsilon}\right)m}{s_2} \leq e^{-\epsilon(s_1+s_2)}\binom{\frac{s_1m}{s_1+s_2}}{s_1}\binom{\frac{s_2m}{s_1+s_2}}{s_2},\] and hence 
\begin{equation}\label{eq:type_b_sets}
\sum_{\calA}^{(b)}\binom{|A_1|}{s_1}\binom{|A_2|}{s_2} \leq \exp(-2^{7}m^{1/6}(s_1+s_2)^{2/3}\sqrt{\log n})\binom{\frac{s_1m}{s_1+s_2}}{s_1}\binom{\frac{s_2m}{s_1+s_2}}{s_2}.
\end{equation}
Finally, it remains to count the relevant $X_1,X_2$ with containers satisfying property~\ref{item:container_close_to_ap}.
Observe that there are at most 
\begin{equation}\label{eq:type_c_sets_1}
\sum_{i=1}^2\sum_{s_i'=8\alpha^{-1} \epsilon(s_1+s_2)}^{s_i}\binom{|A_i|}{s_i-s_i'}\binom{\epsilon|A_i|}{s_i'}\binom{|A_{3-i}|}{s_{3-i}}
\end{equation}
pairs of sets $X_i\subset A_i$ with $|X_i|=s_i$ that violate the theorem statement, since for at least one $\delta\in\{1,2\}$ there must be at least $s_\delta'$ elements in $T_\delta$ for some $s_\delta'\geq 8\alpha^{-1} \epsilon (s_1+s_2)$.
Indeed, otherwise $X_i \setminus T_i \subset P_i$ with $|P_i| \leq \frac{s_im}{s_1+s_2}+4\sqrt{\epsilon}m$ and $|X_i \cap T_i| \leq 8\alpha^{-1} \epsilon (s_1+s_2)$ for both $i$.
For any $d\leq c\leq a/4$, it holds (see for instance~\cite{C2019}) that \[\binom{a}{c-d}\binom{b}{d}\leq \binom{a}{c}\left(\frac{4bc}{ad}\right)^d,\] so applying this to each innermost summand of \eqref{eq:type_c_sets_1} gives
\begin{align*}
\binom{|A_i|}{s_i-s_i'}\binom{\epsilon|A_i|}{s_i'}\binom{|A_{3-i}|}{s_{3-i}} &\leq \left(\frac{4\epsilon s_i}{s_i'}\right)^{s_i'}\binom{|A_1|}{s_1}\binom{|A_2|}{s_2} \\
&\leq \left(\frac{4\epsilon s_i}{s_i'}\right)^{s_i'}\binom{(1+2\epsilon)\frac{s_1m}{s_1+s_2}}{s_1}\binom{(1+2\epsilon)\frac{s_2m}{s_1+s_2}}{s_2}\\
&\leq \left(\frac{4\epsilon s_i}{s_i'}\right)^{s_i'}\left(1+4\alpha^{-1}\epsilon\right)^{s_1+s_2}\binom{\frac{s_1m}{s_1+s_2}}{s_1}\binom{\frac{s_2m}{s_1+s_2}}{s_2}
\end{align*} 
for every $i\in\{1,2\}$ and $s_i\geq s_i'\geq 8\alpha^{-1}\epsilon(s_1+s_2)$.
Here, for the last inequality we used the bound $\binom{a}{c}\leq \left(\frac{a-c}{b-c}\right)^c \binom{b}{c}$ valid for any $a\geq b\geq c\geq 0$, as well as the upper bound $\alpha\leq 1$.
Note that, by our choice of $\epsilon$, we have $\max\{|\calA|,s_1+s_2\} \leq \exp(\epsilon(s_1+s_2))$, hence summing \eqref{eq:type_c_sets_1} over all triples $(A_1,A_2,B)\in \calA$ we obtain
\begin{equation}\label{eq:type_c_sets_2}
\begin{split}
&\sum_{\calA}^{(c)} \sum_{i=1}^2\sum_{s_i'=8\alpha^{-1} \epsilon(s_1+s_2)}^{s_i}\binom{|A_i|}{s_i-s_i'}\binom{\epsilon|A_i|}{s_i'}\binom{|A_{3-i}|}{s_{3-i}}\\
\leq{} &|\calA|(1+4\alpha^{-1}\epsilon)^{s_1+s_2}\binom{\frac{s_1m}{s_1+s_2}}{s_1}\binom{\frac{s_2m}{s_1+s_2}}{s_2}\sum_{i=1}^2 s_i \max_{s_i'\geq 8\alpha^{-1} \epsilon(s_1+s_2)}  \left(\frac{4\epsilon s_i}{s_i'}\right)^{s_i'}\\
\leq{} &\exp\left(6\alpha^{-1}\epsilon (s_1+s_2)\right)2^{-16\alpha^{-1}\epsilon(s_1+s_2)}\binom{\frac{s_1m}{s_1+s_2}}{s_1}\binom{\frac{s_2m}{s_1+s_2}}{s_2}\\
\leq{} &\exp(-4\alpha^{-1}\epsilon(s_1+s_2))\binom{\frac{s_1m}{s_1+s_2}}{s_1}\binom{\frac{s_2m}{s_1+s_2}}{s_2}\\
={} &\exp(-2^{10}m^{1/6}(s_1+s_2)^{2/3}\sqrt{\log n})\binom{\frac{s_1m}{s_1+s_2}}{s_1}\binom{\frac{s_2m}{s_1+s_2}}{s_2}.
\end{split}
\end{equation}

To conclude, note that bounds \eqref{eq:type_a_sets}, \eqref{eq:type_b_sets} and \eqref{eq:type_c_sets_2} imply the probability we claimed in the statement since we can fix a single pair of disjoint arithmetic progressions of length $\frac{s_im}{s_1+s_2}$ respectively with the same common difference and see that any of the $\prod\binom{s_im/(s_1+s_2)}{s_i}$ pairs of $s_i$-subsets will have a sumset of size at most $m$.
\end{proof}

Note that in the proof above, the main restriction concerning the relation between $s_1$ and $s_2$ comes from the application of Theorem~\ref{prop:freiman3k4robust}, our stability result.
Since counting statements using the container framework only require the use of supersaturation and not stability, we thus get the following less restrictive result which is the technical version of Theorem~\ref{thm:easymain_counting}.

\begin{thm}\label{thm:counting_abeliangroups}
Let $G$ be an abelian group.
Let $s_1,s_2,n$ be integers satisfying \[s_2 \geq s_1 \geq \max\left(\sqrt{(s_1+s_2)\log n}, 2^{48}(\log n)^3-s_2\right),\] and let $m$ be an integer such that \[s_1+s_2\leq m\leq \min\left(\frac{s_1^2}{\log n}, \frac{(s_1+s_2)^2}{2^{48}(\log n)^3}\right).\]
Then for any $F_1,F_2\subset G$ with $|F_i|=n$, it holds that the number of pairs of sets $(X_1,X_2)\in 2^{F_1}\times 2^{F_2}$ with $|X_i|=s_i$ and $|X_1+X_2|\leq m$ is at most \[\exp\left(2^{10}m^{1/6}(s_1+s_2)^{2/3}\lambda^{2/3}\sqrt{\log n}\right)\binom{s_1(m+\beta)/(s_1+s_2)}{s_1}\binom{s_2(m+\beta)/(s_1+s_2)}{s_2},\] where $\lambda = \min\left(\frac{m}{m-s_1-s_2},\log(s_1+s_2)\right)$ and $\beta = \beta(m+2^{8}m^{7/6}(s_1+s_2)^{-1/3}\lambda^{-1/3}\sqrt{\log n})$.
\end{thm}

\begin{proof}
We can apply Theorem~\ref{thm:containerfamily} with $s_1,s_2,m,n$ and $1/4>\epsilon>0$ to be specified later, let $\calA$ be the family obtained this way.
So for every pair of sets $(X_1,X_2)\in 2^{F_1}\times 2^{F_2}$ there exists a container triple $(A_1,A_2,B)\in\calA$ such that $X_i\subset A_1$ and $B\subset X_1+X_2$.
Note that if we define $\beta=\beta(m+4\epsilon m)$, it holds true that $|A_1|+|A_2|\leq (1+2\epsilon)(m+\beta)$ for any pair $(A_1,A_2)$ appearing in a container triple in $\calA$.
Indeed, by Theorem~\ref{thm:containerfamily}\ref{item:cont_family_contStruct}, we either have \[|A_1|+|A_2|\leq 2\max|A_i| < 2m/\log n \leq (1+2\epsilon)(m+\beta),\] or there are at most $\epsilon^2 |A_1||A_2|$ pairs $(a_1,a_2)$ such that $a_1+a_2\notin B$, and hence Corollary~\ref{cor:supersat} gives the required bound.
Hence the number of pairs $(X_1,X_2)$ satisfying the theorem hypotheses is at most
\begin{align}\label{eq:counting_bound}
|\calA|\max_{(A_1,A_2,B)\in\calA}\binom{|A_1|}{s_1}\binom{|A_2|}{s_2} &\leq \exp\left(\frac{2^{21}\sqrt{m(\log n)^3}}{\epsilon^{2}}\right)\binom{\frac{s_1(1+2\epsilon)(m+\beta)}{s_1+s_2}}{s_1}\binom{\frac{s_2(1+2\epsilon)(m+\beta)}{s_1+s_2}}{s_2}.
\end{align}
If $m/(m-s_1-s_2)\leq \log(s_1+s_2)$ we can again apply the bound $\binom{a}{c}\leq \left(\frac{a-c}{b-c}\right)^c \binom{b}{c}$ valid for any $a\geq b\geq c\geq 0$ to both binomials in~\eqref{eq:counting_bound} separately and see that it is at most \[\exp\left(2^{21}\epsilon^{-2}\sqrt{m}(\log n)^{3/2}+2\epsilon\lambda(s_1+s_2)\right)\binom{\frac{s_1(m+\beta)}{s_1+s_2}}{s_1}\binom{\frac{s_2(m+\beta)}{s_1+s_2}}{s_2}.\]
Suppose now that $m/(m-s_1-s_2)\geq \log(s_1+s_2)$, and note that this implies in particular $m=s_1+s_2+o(1)$.
We compute 
\begin{align*}
\log\left(\binom{(1+\delta)a}{b}\binom{a}{b}^{-1}\right) &= \sum_{i=0}^{b-1}\log\left(1+\frac{\delta a}{a-i}\right)\\
&\leq \delta a\int_0^{b}(a-x)^{-1}dx \\
&\leq \delta a \log a.
\end{align*}
Applying this with $\delta=2\epsilon$, $a=s_i (m+\beta)/(s_1+s_2)$ and $b=s_i$ and noting that $\beta\leq 2m$, we can upper bound~\eqref{eq:counting_bound} by \[\exp\left(2^{21}\epsilon^{-2}\sqrt{m}(\log n)^{3/2}+2^3\epsilon\lambda(s_1+s_2)\right)\binom{\frac{s_1(m+\beta)}{s_1+s_2}}{s_1}\binom{\frac{s_2(m+\beta)}{s_1+s_2}}{s_2}.\]
Hence setting $\epsilon = 2^6 m^{1/6}(s_1+s_2)^{-1/3}\lambda^{-1/3}\sqrt{\log n}< 1/4$ implies \[\beta(m+4\epsilon m) = \beta(m+2^{8}m^{7/6}(s_1+s_2)^{-1/3}\lambda^{-1/3}\sqrt{\log n})\] and the number of pairs $(X_1,X_2)$ satisfying the theorem hypotheses is at most \[\exp\left(2^{10}m^{1/6}(s_1+s_2)^{2/3}\lambda^{2/3}\sqrt{\log n}\right)\binom{\frac{s_1(m+\beta)}{s_1+s_2}}{s_1}\binom{\frac{s_2(m+\beta)}{s_1+s_2}}{s_2}.\]
\end{proof}

We note that Theorem~\ref{thm:easymain_counting} is obtained via a straightforward application of Theorem~\ref{thm:counting_abeliangroups}.
Finally, similar to the proof of Theorem~\ref{thm:easymain_structure}, when a result comparable to Freiman's $3k-4$ exists one can handle the case $m/(m-s_1-s_2)\geq\log(s_1+s_2)$ differently than it was done in the proof of Theorem~\ref{thm:counting_abeliangroups}.
Essentially, this case can only happen when $m$ is very close to $s_1+s_2$ and hence such a Freiman type result would actually tell us the precise structure instead of just the number of such pairs.

\section{Proof of Theorem~\ref{thm:container}}\label{sec:containerlemmaproof}

\subsection{Setup}
From now on all hypergraphs are allowed to have multi-edges, and the edges are counted with multiplicity.
Let $r,r_0\in \N$, $m\in \N$ and let $R$ be a positive real. Let $b$ be positive integer and suppose that $\calH$ is a $(1,\dots,1,r_0)$-bounded $r$-partite hypergraph with vertex set $V=V_1\cup\dots\cup V_r$ satisfying~\eqref{deg cond} for each vector $y\in \left(\prod_{i=1}^{r-1}\{0,1\}\right)\times \{0,1,\ldots,r_0\}$, $b\leq \min_i |V_i|$ and $b\leq m$ as in the statement of Theorem~\ref{thm:container}.
We also define the vector $w=(w_1,\dots,w_r)=(|V_1|,\dots,|V_{r-1}|,m)$ which will be helpful to reduce notational clutter later. 
Note that using $w$, one can write the degree condition~\eqref{deg cond} as \[\Delta_y(\calH) \leq R\left(\prod_{i=1}^{r}w_i^{y_i}\right)^{-1}b^{|y|-1}e(\calH)\left(\frac{m}{q}\right)^{\mathds{1}[y_r>0]}.\]
We claim that, without loss of generality we may assume that $m \le |V_r|$. Indeed, if $m > |V_r|$, then we may replace $m$ with $|V_r|$ as $\calI_m(\calH) \subseteq \calI_{m'}(\calH)$ for any  $m'\geq m$, and the right-hand side of~\eqref{deg cond} is a non-increasing function in $m$.
We shall be working only with hypergraphs with edge cardinalities coming from the set 
\[
  \calU := \big\{x \in \left(\prod_{i=1}^{r-1}\{0,1\}\right)\times \{1,2,\ldots,r_0\}:~x_i\leq x_{i+1}~\text{and}~ r_0x_{r-1}\leq x_r~\text{for}~ 1\leq i<r\big\}.
\]
The maximum codegrees we must check for each edge size $x\in\calU$ will come from the set
\[
  \calV(x) := \left(\prod_{i=1}^{r}\{0,\ldots,x_i\}\right)\setminus \{(0,\dots,0)\}.
\]
We now define a collection of numbers that will be upper bounds on the maximum codegrees of the hypergraphs constructed by our algorithm. 
To be more precise, for each $x \in \calU$ and all $v\in \calV(x)$, we shall force the maximum $v$-codegree of the $x$-bounded hypergraph not to exceed the quantity $\Del$, defined as follows.

\begin{defn}
  \label{dfn:Delta}
  For every $x \in \calU$ and every $v \in \calV(x)$, we define the number $\Del$ using the following recursion:
  \begin{enumerate}[label=(\arabic*)]
  \item 
    If $x=(1,\dots,1,r_0)$, set $\Del := \Delta_{v}(\calH)$ for all $v \in \calV(x)$.\smallskip
  \item
    Given $x\in \calU$, let $i'=\min \{i: x_i>0\}$ and $x-e_{i'}=x'\in\calU$ where $e_1,\dots,e_r$ are the standard basis vectors of $\R^r$.
    If $v\in \calV(x)$ satisfies $v_{i'}>0$, denote similarly $v-e_{i'}=:v'\in\calV(x')$.
    Note that in this case $i'$ depends on $x$, so $v'$ also depends on $x$, not only on $v$, but we omit it from the notation to avoid clutter.
    Then define
    \[
      \Delta_{v'}^{x'} := \max \left\{ 2 \Delta_{v}^{x}, \, \frac{b}{w_{i'}} \Delta_{v'}^{x} \right\}.
    \]
  \end{enumerate}
\end{defn}

The above recursive definition will be convenient in some parts of the analysis. In other parts, we shall require the following explicit formula for $\Del$, which one easily derives from Definition~\ref{dfn:Delta} using a straightforward induction on $r_0+r-1-|x|$.

\begin{obs}
  \label{obs:Delta}
  For all $x$ and $v$ as in Definition~\ref{dfn:Delta},
  \[
    \Del = \max \left\{ 2^{|z|} \prod_{i=1}^{r-1} \left(\frac{b}{|V_i|}\right)^{1-v_i-z_i} \hspace{-1.2pt} \left(\frac{b}{m}\right)^{r_0-v_r-z_r} \hspace{-1pt} \Delta_{v+z}(\calH) : z\in \left(\prod_{i=1}^{r-1}\{0,1-x_i\}\right)\times [0,r_0-x_r] \right\}.
  \]
\end{obs}
For future reference, we note the following two simple corollaries of Observation~\ref{obs:Delta} and our assumptions on the maximum degrees of $\calH$, see~\eqref{deg cond}. Suppose that $x \in \calU$ such that $i\in[r]$ is the least index with $e_i\in \calV(x)$.
If $i<r$, then by definition of $\calU$ it holds that $x_j=0$ for all $1\leq j<i$, $x_j=1$ for all $i\leq j< r$ and $x_r=r_0$, so 
\begin{equation}
  \label{eq:Delta01}
  \begin{split}  
    \Delta_{e_i}^{x} & \le 2^{i}R \prod_{j=1}^{i-1} \left(\frac{b}{|V_j|}\right) \frac{e(\calH)}{|V_i|}.
  \end{split}  
\end{equation}
If $i=r$, then $x_j=0$ for all $1\leq j<r$ and
\begin{equation}
  \label{eq:Delta10}
  \begin{split}
    \Delta_{e_i}^{x} \le 2^{r+r_0}R \prod_{j=1}^{r-1} \left(\frac{b}{|V_j|}\right) \left(\frac{b}{m}\right)^{r_0-x_r} \frac{e(\calH)}{q}.
  \end{split}
\end{equation}
\par
We will build a sequence of hypergraphs with decreasing maximum edge size, starting with $\calH$, and making sure that for each hypergraph $\calG$ in the sequence we have an appropriate bound on its maximum codegrees. To this end we define the following set of pairs with large codegree.
\begin{defn}
  \label{dfn:MDel}
  Given $x \in \calU$, $v \in \calV(x)$, and an $x$-bounded hypergraph $\calG$, we define
  \[
    \MDel(\calG) = \left\{ L \in \prod_{i=1}^r \binom{V_i}{v_i} : d_\calG(L) \ge \Del/2 \right\}.
  \]
\end{defn}

\subsection{The algorithm}
\label{sec:algorithm}

We shall now define precisely a single round of the algorithm we use to prove the container lemma. To this end, fix some $x \in \calU$, set $i':=\min \{i: x_i>0\}$ and
\begin{equation}
  \label{eq:i-prime}
  x'=x-e_{i'}.
\end{equation}
Suppose that $\calG$ is an $x$-bounded hypergraph with $V(\calG) = V(\calH)$. A single round of the algorithm takes as input an arbitrary $I \in \calI(\calG)$ and outputs an $x'$-bounded hypergraph $\calG_*$ satisfying $V(\calG_*) = V(\calG)$ and $I \in \calI(\calG_*)$ as well as a~set $S\subseteq I\cap V_{i'}$ such that $|S|\leq b$. 
Crucially, the number of possible outputs of the algorithm (over all possible inputs $I \in \calI(\calG)$) is at most $\binom{|V_{i'}|}{\le b}$.

Assume that there is an implicit linear order $\preccurlyeq$ on $V(\calG)$. The \emph{$i'$-maximum vertex} of a hypergraph $\calA$ with $V(\calA) = V(\calG)$ is the $\preccurlyeq$-smallest vertex among all $v\in V_{i'}$ of maximal degree.

\medskip

\noindent
\textbf{The algorithm.}
Set $\calA^{(0)} = \calG$, $S=\emptyset$ and $\calG_*^{(0)}=(V(\calG),\emptyset)$. 
Do the following for each integer $j \ge 0$ in turn:
\begin{enumerate}[label={(S\arabic*)}]
\item
  \label{item:alg-stop}
  If $|S| = b$ or $\calA^{(j)}$ is empty, then set $L = j$ and \STOP.
\item
  \label{item:alg-c-maximum}
  Let $u_j\in V_{i'}$ be the $i'$-maximum vertex of $\calA^{(j)}$.
\item
  \label{item:alg-main-step}
  If $u_j\in I$, then add $j$ to the set $S$ and let
  \[
    \calG_*^{(j+1)} := \calG_*^{(j)} \cup \Big\{ E\setminus\{u_j\} : E \in \calA^{(j)} \text{ and } u_j \in E \Big\}.
  \]
\item
  \label{item:alg-cleanup}
  Let $\calA^{(j+1)}$ be the hypergraph obtained from $\calA^{(j)}$ by removing from it all edges $E$ such that either of the following hold:
  \begin{enumerate}[label={(\alph*)}]
  \item
    \label{item:cleanup-1}
    $u_j \in E$,
  \item
    \label{item:cleanup-2}
    there exists a non-empty $T\subseteq E$, such that 
    \[
      T \in M^{x'}_v\big( \calG_*^{(j+1)} \big)
    \]
 for some $v\in \mathcal{V}(x')$.
  \end{enumerate}
\end{enumerate}
Finally, set $\calA := \calA^{(L)}$ and $\calG_* := \calG_*^{(L)}$. Moreover, set
\[
  W := \big\{ 0, \dotsc, L-1 \big\} \setminus S = \Big\{ j \in \big\{ 0, \dotsc, L-1 \big\} : u_j\not \in I\Big\}.
\]

\smallskip

Observe that the algorithm always stops after at most $|V(\calG)|$ iterations of the main loop. 
Indeed, since all hyperedges $E$ with $u_j \in E$ are removed from $\calA^{(j+1)}$ in part~\ref{item:cleanup-1} of step~\ref{item:alg-cleanup}, the vertex $u_j$ cannot be the $i'$-maximum vertex of any $\calA^{(j')}$ with $j' > j$ and hence the map $\{0, \dotsc, L-1\} \ni j \mapsto u_j \in V(\calG)$ is injective. 

\subsection{The analysis}

We shall now establish some basic properties of the algorithm described in the previous subsection. To this end, let us fix some $x \in \calU$, $x'$ and $i'$ as defined in~\eqref{eq:i-prime}. Moreover, suppose that $\calG$ is an $x$-bounded hypergraph and that we have run the algorithm with input $I \in \calI(\calG)$ and obtained the $x'$-bounded hypergraph $\calG_*$, the integer $L$, the injective map $\{0, \dotsc, L-1\} \ni j \mapsto u_j \in V(\calG)$, and the partition of $\{0, \dotsc, L-1\}$ into $S$ and $W$ such that $u_j\in I$ if and only if $j \in S$. We first state two straightforward, but fundamental, properties of the algorithm.

\begin{obs}
  \label{obs:h-in-FF-Gn}
  If $I \in \calI(\calG)$, then $I \in \calI(\calG_*)$.
\end{obs}

\begin{proof}
  Observe that $\calG_*$ contains only edges of the form $E\setminus \{v\}$ where $v\in E\cap I$ and $E \in \calG$, see~\ref{item:alg-main-step}. Hence, if $I$ contained the edge $E\setminus \{v\}$ it would also contain the edge $E$.
\end{proof}

The next observation says that if the algorithm applied to two sets $I$ and $I'$ outputs the same set $\{u_j : j \in S\}$, then the rest of the output is also the same. 

\begin{obs}
  \label{obs:number-of-containers}
Fix the hypergraph $\calG$ we input in the algorithm, suppose that the algorithm applied to $I' \in \calF(\calG)$ outputs a hypergraph~$\calG_*'$, an integer $L'$, a map $j \mapsto u_j'$, and a partition of $\{0, \dotsc, L'-1\}$ into $S'$ and $W'$. If $\{u_j : j \in S\} = \{u_j' : j \in S'\}$, then $\calG_* = \calG_*'$, $L = L'$, $u_j = u_j'$ for all $j$, and $W = W'$.
\end{obs}

\begin{proof}
  The only step of the algorithm that depends on the input pair $I$ is~\ref{item:alg-main-step}. There, an index $j$ is added to the set $S$ if and only if $u_j\in I$. Therefore, the execution of the algorithm depends only on the set $\{u_j : j \in S\}$ and the hypergraph $\calG$.
\end{proof}

The next two lemmas will allow us to maintain suitable upper and lower bounds on the degrees and densities of the hypergraphs obtained by applying the algorithm iteratively. The first lemma, which is the easier of the two, states that if all the maximum degrees of $\calG$ are appropriately bounded, then all the maximum degrees of $\calG_*$ are also appropriately bounded.

\begin{lem}\label{lemma:alg-analysis-degrees}
Given $v \in \calV(x)$ with $v_{i'} > 0$, let $v'=v-e_{i'}$. If $\Delta_{v}(\calG) \le \Del$, then $\Delta_{v'}(\calG_*) \le \Delp$.
\end{lem}

\begin{proof}
Suppose (for a contradiction) that there exists a set $T$, with $|T\cap V_i| = v'_i$ for all $i$, such that $\deg_{\calG_*}(T) > \Delp$. Let $j$ be the smallest integer satisfying
  \[
    \deg_{\calG_*^{(j+1)}}(T) > \Delp/2
  \]
  and note that $j \ge 0$, since $\calG_*^{(0)}$ is empty. We claim first that
   \begin{equation}
    \label{eq:deg-j-is-deg-final}
    \deg_{\calG_*} (T) = \deg_{\calG_*^{(j+1)}}(T).
  \end{equation}
  Indeed, observe that $T \in M^{x'}_{v'} \big( \calG_*^{(j+1)} \big)$, and therefore the algorithm removes from $\calA^{(j)}$ (when forming $\calA^{(j+1)}$ in step~\ref{item:alg-cleanup}) all edges $E$ such that $T\subset E$. As a consequence, no further edges $E$ with $T\subseteq E$ are added to $\calG_*$ in step~\ref{item:alg-main-step}.

  We next claim that
  \begin{equation}
    \label{eq:one-step-deg-change}
    \deg_{\calG_*^{(j+1)}}(T) - \deg_{\calG_*^{(j)}}(T) \le \Del.
  \end{equation}
  To see this, recall that when we extend $\calG_*^{(j)}$ to $\calG_*^{(j+1)}$ in step~\ref{item:alg-main-step}, we only add edges $E\setminus \{u_j\}$ such that $E \in \calA^{(j)} \subseteq \calG$ and $u_j \in E$. Therefore, setting $T^*=T\cup \{u_j\}$, we have
  \begin{equation*}
    \deg_{\calG_*^{(j+1)}}(T) - \deg_{\calG_*^{(j)}}(T) \le \deg_{\calG}(T^*) \le \Delta_{v}(\calG) \le \Del,   
  \end{equation*}
  where the last inequality is by our assumption, as claimed.

  Combining~\eqref{eq:deg-j-is-deg-final} and~\eqref{eq:one-step-deg-change}, it follows immediately that
  \[
    \deg_{\calG_*}(T) \le \Delp/2 + \Del \le \Delp,
  \]
  where the final inequality holds by Definition~\ref{dfn:Delta}. This contradicts our choice of $T$ and therefore the lemma follows.
\end{proof}

We are now ready for the final lemma, which is really the heart of the matter. We will show that if $\calG$ has sufficiently many edges and all of the maximum degrees of $\calG$ are appropriately bounded, then either the output hypergraph~$\calG_*$ has sufficiently many edges, or the output set $W$ must be big. We remark that here we shall use the assumption that $|I\cap V_r|\geq |V_r|-m$.

\begin{lem}
  \label{lemma:alg-analysis-progress}
  Suppose that $|I\cap V_r|\geq |V_r|-m$ and let $\alpha > 0$. If 
  \begin{enumerate}[label=(A\arabic*)]
  \item
    \label{item:assumption-edges}
    $e(\calG) \ge \alpha \prod_{i=1}^{r-1}\left(\frac{b}{|V_i|}\right)^{1-x_i}\left(\frac{b}{m}\right)^{r_0-x_r} e(\calH)$ and\smallskip
     \item
    \label{item:assumption-Delta}
    $\Delta_{v}(\calG) \le \Del$ for every $v \in \calV(x)$,\smallskip
  \end{enumerate}
  then at least one of the following statements is true:
  \begin{enumerate}[label=(P\arabic*)]
  \item
    \label{item:reduce-uniformity}
    $e(\calG_*) \ge 2^{-|x|-x_r-1} \alpha \prod_{i=1}^{r-1}\left(\frac{b}{|V_i|}\right)^{1-x'_i}\left(\frac{b}{m}\right)^{r_0-x'_r} e(\calH)$.\smallskip
  \item
    \label{item:determine-many-0}
    $i'<r$ and $|W| \ge 2^{-i'-1} R^{-1} \alpha |V_{i'}|$.\smallskip
  \item
    \label{item:determine-many-1}
    $i' = r$ and $|W| \ge 2^{-r-r_0-1} R^{-1} \alpha  q$.
  \end{enumerate}
\end{lem}

\begin{proof}
  Recall that $\calG_*$ (and $\calG_*^{(j)}$ etc.) are multi-hypergraphs and that edges are counted with multiplicity.
  We observe that
  \begin{equation}
    \label{eq:eG-sum}
    e(\calG_*) = \sum_{j \in S} \left( e(\calG_*^{(j+1)}) - e(\calG_*^{(j)}) \right) = \sum_{j \in S} \Delta_{e_{i'}}(\calA^{(j)}),
  \end{equation}
  since $e(\calG_*^{(j+1)}) - e(\calG_*^{(j)}) = d_{\calA^{(j)}}(\{u_j\})$ and $u_j$ is the $i'$-maximum vertex of $\calA^{(j)}$ for each $j \in S$, and $\calG_*^{(j+1)} = \calG_*^{(j)}$ for each $j \not\in S$. 
  To bound the right-hand side of~\eqref{eq:eG-sum}, we count the edges removed from $\calA^{(j)}$ in \ref{item:cleanup-1} and \ref{item:cleanup-2} of step \ref{item:alg-cleanup}, which gives
  \[
    e(\calA^{(j)}) - e(\calA^{(j+1)}) \le \Delta_{e_{i'}}(\calA^{(j)}) + \sum_{v} \big| \MDelp(\calG_*^{(j+1)}) \setminus \MDelp(\calG_*^{(j)}) \big| \cdot \Delta_{v}(\calG).
  \]
  Summing over $j \in \{0, \ldots, L-1\}$ it follows (using~\eqref{eq:eG-sum}) that
  \[
    e(\calG) - e(\calA) \le e(\calG_*) + |W| \cdot \Delta_{e_{i'}}(\calG) + \sum_{v} \big| \MDelp(\calG_*) \big| \cdot \Del,
  \]
  since $\calA = \calA^{(L)} \subseteq \dots \subseteq \calA^{(0)} =\calG$ and $\Delta_{v}(\calG) \le \Del$ by~\ref{item:assumption-Delta}.
  Furthermore,
  \begin{equation}
    \label{eq:e-G0-increment-estimate}
    \Delta_{e_{i'}}(\calA) \le \Delta_{e_{i'}}(\calA^{(j)}) \le \Delta_{e_{i'}}(\calG) \le \Delta_{e_{i'}}^{x},
  \end{equation}
  since $\calA \subseteq \calA^{(j)} \subseteq \calG$ and $\calG$ satisfies~\ref{item:assumption-Delta}, which implies
  \begin{equation}
    \label{eq:Gn-final-0}
    e(\calG) - e(\calA) \le e(\calG_*) + |W| \Delta_{e_{i'}}^{x} + \sum_{v} \big|\MDelp(\calG_*) \big| \Del.
  \end{equation}
  Combining~\eqref{eq:eG-sum} and~\eqref{eq:e-G0-increment-estimate}, we get
  \begin{equation}
    \label{eq:e-Gn-S}
    e(\calG_*) = \sum_{j \in S} \Delta_{e_{i'}}\big(\calA^{(j)}\big) \ge |S| \Delta_{e_{i'}}(\calA) = b \Delta_{e_{i'}}(\calA),
  \end{equation}
  where the equality is due to the fact that $|S| \neq b$ only when $\calA$ is empty, see step~\ref{item:alg-stop}. 
  
 Next, to bound the sum in~\eqref{eq:Gn-final-0}, observe that, by Definition~\ref{dfn:MDel}, we have
 \[
   \big| \MDelp(\calG_*) \big| \Delnp/2 \le \sum_{T:~|T\cap V_i|=v_i} \deg_{\calG_*}(T) \leq  \binom{x_r}{v_r}e(\calG_*)\leq 2^{x_r}e(\calG_*)
 \]
 for each $v\in \calV(x')$ and therefore
 \begin{equation}
   \label{eq:e-Gn-MDel-final}
   \begin{split}
     \sum_{v\in \calV(x')} \big| \MDelp(\calG_*) \big| \Del & \le 2^{x_r+1} \sum_{v} e(\calG_*)  \left(\Del / \Delnp\right)  \\
     & \le 2^{x_r+1} \big( 2^{|x'|} - 1 \big)  e(\calG_*)  \max_{v} \left\{\Del / \Delnp\right\}\\
     &\leq 2^{x_r+1} \big( 2^{|x'|} - 1 \big)  e(\calG_*)w_{i'} / b,
   \end{split}
 \end{equation}
 where the last inequality follows from Definition~\ref{dfn:Delta}.

 Suppose first that $i'<r$ and observe that substituting~\eqref{eq:e-Gn-MDel-final} into~\eqref{eq:Gn-final-0} yields
  \begin{equation}
    \label{eq:Gn-final-case1}
    e(\calG) - e(\calA) \le e(\calG_*) + |W| \Delta_{e_{i'}}^{x} + 2^{x_r+1} \big( 2^{|x'|} - 1 \big)  e(\calG_*)|V_{i'}| / b.
  \end{equation}
  Moreover, by~\eqref{eq:e-Gn-S} we have
  \begin{equation}
    \label{eq:Delta-estimates-1}
    \frac{e(\calG_*)}{b} \ge \Delta_{e_{i'}}(\calA) \ge \frac{e(\calA)}{|V_{i'}|} 
  \end{equation}
  since the maximum degree of a hypergraph is at least as large as its average degree. Combining~\eqref{eq:Gn-final-case1} and~\eqref{eq:Delta-estimates-1}, we obtain
  \begin{equation}
    \label{eq:eGG-eGGn-W-1}
    \begin{split}
      e(\calG) & \le e(\calG_*) \frac{|V_{i'}|}{b}  \left(\frac{b}{|V_{i'}|} + 1 + 2^{x_r+|x'|+1} - 2\right) + |W|  \Delta_{e_{i'}}^{x} \\
      & \le e(\calG_*)  \frac{|V_{i'}|}{b} 2^{x_r+|x|} + |W| \Delta_{e_{i'}}^{x},
    \end{split}
  \end{equation}
  since $b \le |V_{i'}|$. Now, if the first summand on the right-hand side of~\eqref{eq:eGG-eGGn-W-1} exceeds $e(\calG) / 2$, then~\ref{item:assumption-edges} implies~\ref{item:reduce-uniformity}. Otherwise, the second summand is at least $e(\calG)/2$ and by~\ref{item:assumption-edges} and~\eqref{eq:Delta01},
  \[
    |W| \ge \frac{e(\calG)}{2 \Delta_{e_{i'}}^{x}} \ge \frac{\alpha}{2^{i'+1}R} |V_{i'}|,
  \]
  which is~\ref{item:determine-many-0}.
 
  Finally, suppose $i'=r$.
  Substituting~\eqref{eq:e-Gn-MDel-final} into~\eqref{eq:Gn-final-0} yields, using the bound $\Del / \Delnp \le m / b$,
 \begin{equation}
    \label{eq:Gn-final-case2}
    e(\calG) - e(\calA) \le e(\calG_*) + |W| \Delta_{e_r}^{x} + \big( 2^{x_r+|x|} - 2^{x_r+1} \big) e(\calG_*) \frac{m}{b}.
  \end{equation}
   We claim that
  \begin{equation}
    \label{eq:Delta-estimates-0}
    \frac{e(\calG_*)}{b} \ge \Delta_{e_r}(\calA) \ge \frac{e(\calA)}{m}.
  \end{equation}
 The first inequality follows from~\eqref{eq:e-Gn-S}, so we only need to prove the second inequality. To do so, since $I\in \calF(\calG)$ is an independent set in $\calA$ (and all edges of $\calA$ are contained in $V_r$) then every edge in $\calA$ must be incident to $V_r\setminus I$, which has size at most $m$ by assumption. This shows that \[\Delta_{e_r}(\calA) \ge  \frac{e(\calA)}{|V_r\setminus I|}\ge\frac{e(\calA)}{m}.\]
 Combining~\eqref{eq:Gn-final-case2} and~\eqref{eq:Delta-estimates-0}, we obtain
  \begin{equation}
    \label{eq:eGG-eGGn-W-0}
    \begin{split}
      e(\calG) & \le e(\calG_*) \frac{m}{b} \left(\frac{b}{m} + 1 + 2^{x_r+|x|} - 2^{x_r+1}\right) + |W| \Delta_{e_r}^{x} \\
      & \le e(\calG_*) \frac{m}{b}2^{x_r+|x|} + |W|  \Delta_{e_r}^{x},
    \end{split}
  \end{equation}
  since $b \le m$. Now, if the first summand on the right-hand side of~\eqref{eq:eGG-eGGn-W-1} exceeds $e(\calG) / 2$, then~\ref{item:assumption-edges} implies~\ref{item:reduce-uniformity}. Otherwise, the second summand is at least $e(\calG)/2$ and by~\ref{item:assumption-edges} and~\eqref{eq:Delta10},
  \[
    |W| \ge \frac{e(\calG)}{2\Delta_{e_r}^{x}} \ge \frac{\alpha}{2^{r_0+r+1}R} q,
  \]
  which is~\ref{item:determine-many-1}.
\end{proof}

\subsection{Construction of the container}
\label{sec:constr-container}

In this section, we present the construction of containers for pairs in $\calI_{m}(\calH)$ and analyse their properties, thus proving Theorem~\ref{thm:container}. For each $s \in \{0, \ldots, r_0+r-1\}$, define
\[
  \alpha_s = 2^{-s(2r_0+r)} \qquad \text{and} \qquad \beta_s = \alpha_s \prod_{j=1}^{\min\{{r-1,s\}}}\left(\frac{b}{|V_j|}\right)\left(\frac{b}{m}\right)^{\max\{0,s-r+1\}}.
\]
Given an $I \in \calI_m(\calH)$, we construct the container $(A_1,\dots,A_r)$ for $I$ using the following procedure.

\medskip
\noindent
\textbf{Construction of the container.}
Initialize $s=0,\, x=(1,\dots,1,r_0)$, $\calH^{x} = \calH$ and $S_i = \emptyset$ for all $i\in[r]$.
\begin{enumerate}[label=(C\arabic*)]
\item
  Let $i'$ and $x'$ be defined from $x$ as before.
\item
  \label{item:constr-run-algorithm}  
  Run the algorithm with $\calG \leftarrow \calH^{x}$ to obtain the $x'$-bounded hypergraph $\calG_*$, the sequence $u_0, \ldots, u_{L-1} \in V(\calH)$, and the partition $\{0, 1, \ldots, L-1\} = S \cup W$.
\item
  \label{item:constr-update-signature}
  Let $S_{i'}\leftarrow S_{i'}\cup \{u_j : j \in S\}$.
\item
  \label{item:constr-determine-many}
  If $e(\calG_*) < \beta_{s+1} e(\calH)$, then define $(A_1,\ldots,A_r)$, the container for $I$, by
  \[
    A_{i'} = V_{i'}\setminus\{u_j : j\in W\}
  \]
  and $A_j=V_j$ for $j\neq i'$ and \STOP.
\item
  \label{item:constr-reduce-uniformity}
 % If $e(\calG_*) \ge \beta_{s+1} \cdot e(\calH)$, then 
  Otherwise, let $\calH^{x} \leftarrow \calG_*$, $x \leftarrow x'$ and $s \leftarrow s+1$ and \texttt{CONTINUE}.
\end{enumerate}

We will show that the above procedure indeed constructs containers for $\calI_m(\calH)$ that have the desired properties. To this end, we first claim that for each $x \in \calU \cup \{0\}$, the hypergraph $\calH^{x}$, if it was defined, satisfies:
\begin{enumerate}[label=(\textit{\roman*})]
\item
  \label{item:container-propty-1}
  $I\in \calI(\calH^{x})$ and
\item
  \label{item:container-propty-2}
  $\Delta_{v}(\calH^{x}) \le \Del$ for every $v \in \calV(x)$.
\end{enumerate}
Indeed, one may easily prove~\ref{item:container-propty-1} and~\ref{item:container-propty-2} by induction on $|x| - |v|$. The base case is true by Definition~\ref{dfn:Delta}, and the inductive step follows immediately from Observation~\ref{obs:h-in-FF-Gn} and Lemma~\ref{lemma:alg-analysis-degrees}.

Secondly, we claim that for each input $I \in \calI_{m}(\calH)$, step~\ref{item:constr-determine-many} is called for some $s$ and hence the container $ (A_1,\ldots,A_r)$  is defined. If this were not true, the condition in step~\ref{item:constr-reduce-uniformity} would be met $r+r_0-1$ times and, consequently, we would finish with a non-empty $(0,\dots,0)$-bounded hypergraph $\calH^{0}$, i.e., we would have $\emptyset\in E(\calH^0)$. But this contradicts~\ref{item:container-propty-1}, since $\emptyset\subset I$, so it would not be independent.

Suppose, therefore, that step~\ref{item:constr-determine-many} is executed when $\calG = \calH^{x}$ for some $x \in \calU$. We claim that $e(\calH^{x}) \ge \beta_s e(\calH)$. This is trivial if $s = 0$ since here $\calH^x=\calH$, and for $s > 0$ it holds since otherwise step~\ref{item:constr-determine-many} would have been executed in the previous iteration. We therefore have
\[
  e(\calG) = e(\calH^{x}) \ge \beta_s e(\calH) \qquad \text{and} \qquad e(\calG_*) < \beta_{s+1} e(\calH),
\]
which, by Lemma~\ref{lemma:alg-analysis-progress} and~\ref{item:container-propty-2}, implies that either~\ref{item:determine-many-0} or~\ref{item:determine-many-1} of Lemma~\ref{lemma:alg-analysis-progress} holds. Define $\delta = 2^{-(r_0+r-1)(2r_0+r)}R^{-1}$ and note that $\delta \leq \alpha_sR^{-1}$ for all $s\in[0,r_0+r-1]$.
If $i'<r$, we see that \ref{item:determine-many-0} implies
\[
  |W| \ge 2^{-r-1}R^{-1} \alpha_s |V_{i'}| \ge \alpha_r R^{-1} |V_{i'}| \geq \delta |V_{i'}|,
\]
Similarly, if $i' = r$, then by \ref{item:determine-many-1},
\[
  |W| \ge 2^{-r_0-r-1}R^{-1} \alpha_s q \ge \alpha_{r_0+r-1}R^{-1} q = \delta q.
\]
This verifies that $(A_1,\ldots,A_r)$ satisfies property~\ref{item:container-2} from the statement of Theorem~\ref{thm:container}.

Let $\calS$ denote the set of all tuples $(S_1,\dots,S_r)$ that were defined in \ref{item:constr-update-signature} when running the procedure for all $I\in\calI_m(\calH)$.
We define $g(I)=(S_1,\dots,S_r)$ and $f(g(I))=(A_1,\dots,A_r)$, where $(A_1,\dots,A_r)$ is the container tuple that was defined in \ref{item:constr-determine-many}.
Note that $f$ is well-defined by Observation~\ref{obs:number-of-containers}.
This follows directly when $i'<r$, since here the set $S_{i'}$ is equivalent to the set $S$ obtained in \ref{item:constr-run-algorithm}.
But then, in particular, everything will be the same the first time that $i'=r$, and hence $r$--maximum vertices will be considered at the same time.

Finally, we see that clearly the inclusion statements of properties~\ref{item:container-1} and~\ref{item:container-3} hold by construction, and the second one in \ref{item:container-3} is true since every $S_i$ starts empty and we stop as soon as \ref{item:constr-determine-many} is true for the first time. \qed

\section{Concluding Remarks}\label{sec:concluding}

In this paper we have focussed on the asymmetric version of the typical structure of a pair of sets subject to a constraint on the size of their sumset. 
As explained in the introduction, this formulation of the result is very natural with respect to the perspective of classical results on sumsets like the Brunn--Minkowski inequality or Kneser's theorem. 
As illustrated by this paper, obtaining such asymmetric versions of results in additive combinatorics, which is most often concerned with studying a single set possessing some additive structure, can be  a nontrivial task. 
It is worth noting that our result as is does not supersede but rather  complements the one by Campos~\cite{C2019} for the case $A=B$:
The latter cannot be directly derived from ours, as here the estimations leading to an almost surely result are based on the random choice of two independent sets rather than only one. 
Still, the quantitative bounds of our result are comparable when considering sets of the same size.

Even if a simpler version of the container result, Theorem~\ref{thm:container}, would have been sufficient for the proof of our main result, we have chosen to state it in more generality, at no cost, as in this wider generality  the result  may be suitable to address a number of additional applications, which nevertheless require developments which do not fit the length of a single paper. 

It is possible to prove a straight-forward generalization of Theorem~\ref{thm:containerfamily} for arbitrary fixed $h$ using the same arguments that were used in the $h=2$ case, but there are some caveats in the specifics.
Still, in light of this fact, a natural step further is to handle multiple set addition: if $A_1,\ldots , A_k$ is a family of independently chosen random sets of integers with cardinality $s$ among those satisfying a constraint on their sumset, say $|A_1+\cdots +A_l|\le Ks$, then with high probability each of the sets is almost contained in an arithmetic progression of size $Ks/l$ having the same common difference.
The quantitative aspects of the statement depend on the strength of supersaturation and stability results analogous to Corollaries~\ref{cor:supersat} and~\ref{cor:relativestability}, which are only partially existing in the literature. 
Tight bounds on the cardinality of multiple set addition can be found for example in Lev~\cite{L1996}, but to our knowledge no inverse results deriving the structure of sets in this multiple addition setting are available. 
A result of this kind, both for the analogue of the $3k-4$ theorem and its robust version might be attained along the lines of known results for the addition of two distinct sets by the use of the multiple set addition version of Kneser's theorem given by DeVos, Goddyn, and Mohar~\cite{DGM2009}. 
This opens a path to be yet explored.

A second natural direction is to translate the structural result into the more general setting of arbitrary abelian groups. 
The parameter $\beta(t)$ measuring the size of the largest subgroup with cardinality up to $t$ is a useful one when we consider the counting version of our main result. 
We have included its asymmetric version in this paper, Theorem~\ref{thm:easymain_counting}.  
The structural description given by Green and Ruzsa~\cite{GR2007} in their extension of Freiman's theorem to general abelian groups provides a guideline for the structure of typical sets with bounded sumset in general abelian groups. 
As explored in this paper as well as in~\cite{C2019} and~\cite{CCMMS2021}, the typical structure of sets with bounded sumset in the integers, in this case plain arithmetic progressions, is a lot simpler than the general one given by Freiman's theorem, so one would expect the same to be the case in more general groups.
The recently obtained robust version of the Balog-Szemer\' edi-Gowers theorem by Shao~\cite{S2019}, combined with an appropriate Freiman $3k-4$ type theorem seems to be the correct set of tools to achieve this objective.  
An important natural case to explore is that of groups of prime order, $G=\Z/p\Z$, where sufficiently strong analogues of the Freiman $3k-4$ Freiman theorem are already available and an equivalent statement of our main result might require less work to prove.  

An even less explored direction is to also translate the structural result to general groups, not necessarily abelian. 
Perhaps the more appropriate quest in this setting is to ask for the typical structure of approximate groups in the light of its structural characterization by Breuillard, Green and Tao~\cite{BGT2012}. 

Coming back to more classical problems in additive combinatorics, let us conclude by mentioning a natural example which can illustrate the use of the techniques used in this paper, the case of $A+B$ when $B=\lambda\ast A$, the \emph{dilation} of $A$ by some factor $\lambda$. 
It is known that $A+\lambda\ast A$ has size at least $(\lambda+1)|A|$ and, for  $\lambda$ prime, the structure of extremal sets is known~\cite{CHS2009}. 
Can one prove an approximate structure result similar in scope to Theorem~\ref{thm:easymain_structure} for this specific problem?

{\small
\providecommand{\MR}[1]{}
\providecommand{\bysame}{\leavevmode\hbox to3em{\hrulefill}\thinspace}
\providecommand{\MR}{\relax\ifhmode\unskip\space\fi MR }
% \MRhref is called by the amsart/book/proc definition of \MR.
\providecommand{\MRhref}[2]{%
  \href{http://www.ams.org/mathscinet-getitem?mr=#1}{#2}
}
\providecommand{\href}[2]{#2}
}

\section*{Acknowledgments}
{\setlength\parindent{0pt}
Marcelo Campos is partially supported by CNPq.\\
Matthew Coulson is supported by the Spanish Ministerio de Econom\'a y Competitividad through the project MTM2017-82166-P and the FPI-scholarship PRE2018-083621.\\
Oriol Serra is supported by the Spanish Agencia Estatal de Investigaci\'on under project PID2020-113082GB-I00.\\
This research was conducted while Maximilian W\"otzel was a member of the Barcelona Graduate School of Mathematics (BGSMath) as well as the Universitat Polit\'ecnica de Catalunya. M.~W. also acknowledges financial support from the Fondo Social Europeo and the Agencia Estatal de Investigaci\'on through the FPI grant number MDM-2014-0445-16-2 and the Spanish Ministry of Economy and Competitiveness, through the Mar\'ia de Maeztu Programme for Units of Excellence in R\&D (MDM-2014-0445), through the project MTM2017-82166-P, as well as from the Dutch Science Council (NWO) through the grant number OCENW.M20.009.}

\appendix

\section{Proof of Lemma~\ref{lem:binom_prod_bound}}\label{sec:largedeviationbound}

Dividing by the binomial coefficients on the right hand side of~\eqref{eq:binom_prod_bound} and taking the logarithm, we need to prove
\begin{equation}\label{eq:binom_prod_sum}
\sum_{i=0}^{t-1} \log\bigg( 1 - \frac{(2\sqrt{\epsilon}-2\epsilon)m}{\frac{tm}{s+t}-i} \bigg) + \sum_{j=0}^{s-1} \log\bigg( 1 + \frac{2\sqrt{\epsilon}m}{\frac{sm}{s+t}-j} \bigg) \leq -\epsilon(s+t).
\end{equation}
By using the bound $\log(1+x)\leq x-\frac{x^2}{2}+\frac{x^3}{3}$ valid on the interval $(-1,\infty)$, it suffices to prove the upper bound in~\eqref{eq:binom_prod_sum} for the expression
\begin{equation}\label{eq:binom_prod_taylor}
\begin{split}
&-2(\sqrt{\epsilon}-\epsilon)\sum_{i=0}^{t-1}\left(\frac{t}{s+t}-\frac{i}{m}\right)^{-1} 
- \frac{(2\sqrt{\epsilon}-2\epsilon)^2}{2} \sum_{i=0}^{t-1}\left(\frac{t}{s+t}-\frac{i}{m}\right)^{-2} \\
&+2\sqrt{\epsilon}\sum_{j=0}^{s-1}\left(\frac{s}{s+t}-\frac{j}{m}\right)^{-1}
-2\epsilon \sum_{i=0}^{s-1}\left(\frac{s}{s+t}-\frac{j}{m}\right)^{-2} 
+4\epsilon^{3/2}\sum_{i=0}^{s-1}\left(\frac{s}{s+t}-\frac{j}{m}\right)^{-3}.
\end{split}
\end{equation}
Next, we are going to approximate the sums in~\eqref{eq:binom_prod_taylor} by integrals, using the bounds \[\int_{a}^{b} f(x) - \frac{2(b-a)||f||_{\infty}}{n} \leq \frac{1}{n} \sum_{i=na}^{nb}f\left(\frac{i}{n}\right) \leq \int_{a}^{b} f(x)\] which hold for any continuous, non-decreasing function $f$ on the interval $[a,b]$.
We start with the linear terms.
Defining $K>(1+\alpha)$ by $m=K(s+t)$, we get
\begin{equation}\label{eq:lin_term_t}
\begin{split}
-2(\sqrt{\epsilon}-\epsilon)\sum_{i=0}^{t-1}\frac{1}{\frac{t}{s+t}-\frac{i}{m}}
&\leq -2(\sqrt{\epsilon}-\epsilon) \left(m \int_{0}^{t/m}\frac{1}{\frac{t}{s+t}-x}dx - 2\frac{t}{m}\left(\frac{t}{s+t}-\frac{t}{m}\right)^{-1} \right)\\
&= -2(\sqrt{\epsilon}-\epsilon)\left(-m\log\left(1-\frac{s+t}{m}\right)-\frac{2}{K-1}\right)\\
&\leq -2m(\sqrt{\epsilon}-\epsilon)\log\left(\frac{K}{K-1}\right) + \frac{4\sqrt{\epsilon}}{K-1},
\end{split}
\end{equation}
and similarly
\begin{equation}\label{eq:lin_term_s}
\begin{split}
2\sqrt{\epsilon}\sum_{j=0}^{s-1}\left(\frac{s}{s+t}-\frac{j}{m}\right)^{-1} 
&\leq 2m\sqrt{\epsilon}\log\left(\frac{K}{K-1}\right).
\end{split}
\end{equation}
For the quadratic terms, we see that
\begin{equation}\label{eq:quad_term_t}
\begin{split}
- \frac{(2\sqrt{\epsilon}-2\epsilon)^2}{2} \sum_{i=0}^{t-1}\left(\frac{t}{s+t}-\frac{i}{m}\right)^{-2} 
&\leq - \frac{(2\sqrt{\epsilon}-2\epsilon)^2}{2}\left(\frac{K(s+t)^2}{t(K-1)}-\frac{2K(s+t)}{t(K-1)^2}\right)\\
&\leq -\frac{2\epsilon K(s+t)^2}{t(K-1)}+\frac{4\epsilon^{3/2}K(s+t)^2}{t(K-1)}+\frac{8\epsilon K(s+t)}{t(K-1)^2},
\end{split} 
\end{equation}
and similarly for the one involving $s$,
\begin{equation}\label{eq:quad_term_s}
\begin{split}
-2\epsilon \sum_{i=0}^{s-1}\left(\frac{s}{s+t}-\frac{j}{m}\right)^{-2} 
\leq -\frac{2\epsilon K(s+t)^2}{s(K-1)} + \frac{4\epsilon K(s+t)}{s(K-1)^2}.
\end{split}
\end{equation}
Finally, for the cubic term we see that 
\begin{equation}\label{eq:cube_term}
\begin{split}
4\epsilon^{3/2}\sum_{i=0}^{s-1}\left(\frac{s}{s+t}-\frac{j}{m}\right)^{-3}
&\leq \frac{2\epsilon^{3/2}K(2K-1)(s+t)^3}{s^2(K-1)^2}.
\end{split}
\end{equation}
Note that the $2m\sqrt{\epsilon}$ parts of the linear terms cancel out, while \[2\epsilon K(s+t)\log\left(\frac{K}{K-1}\right) \leq \frac{2\epsilon K(s+t)^2}{\max(s,t)(K-1)},\] which follows from the fact that $x\log(1+x^{-1})\leq 1$ for all $x>0$.
On the other hand, because of $(s+t)\geq 2^{5}\alpha^{-1}\geq 2^5(K-1)^{-1}$ and the bounds on $\epsilon$, the sum of all remaining positive term in Equations~\eqref{eq:lin_term_t}--\eqref{eq:cube_term} can be upper bounded by $\frac{\epsilon K(s+t)^2}{\min(s,t)(K-1)}$, and hence we see that \[\eqref{eq:binom_prod_taylor} \leq -\frac{\epsilon K(s+t)^2}{\min(s,t)(K-1)} \leq -\epsilon (s+t),\] which implies~\eqref{eq:binom_prod_sum} and hence proves the statement. \hfill\qedsymbol

\end{document}